\documentclass[11pt,A4]{article}
\usepackage{amsmath,amsthm,amsfonts,amssymb,amscd, amsxtra,color}
\usepackage[active]{srcltx}
\usepackage{url}
\usepackage{cite}
\usepackage{multirow}
\usepackage[margin=2.5cm,nohead]{geometry}
\usepackage{geometry}
\usepackage{color}
\usepackage{float}
\usepackage{textcomp}
\usepackage{amsthm}
\usepackage{amsmath}
\usepackage{amssymb}
\usepackage[ruled]{algorithm2e}

 \usepackage{colortbl,hhline,color,soul,url}

\usepackage{changes}

\usepackage{mathtools}

\newtheorem{theorem}{Theorem}
\newtheorem{lemma}[theorem]{Lemma}
\newtheorem{corollary}[theorem]{Corollary}
\newtheorem{proposition}[theorem]{Proposition}
\newtheorem{remark}[theorem]{Remark}
\newtheorem{definition}[theorem]{Definition}
\newtheorem{example}[theorem]{Example}

\newtheorem{strategy}{Strategy}
\newcommand{\beq}{\begin{equation}}
\newcommand{\eeq}{\end{equation}}
\newcommand{\beqa}{\begin{eqnarray}}
\newcommand{\eeqa}{\end{eqnarray}}
\newcommand{\beqas}{\begin{eqnarray*}}
\newcommand{\eeqas}{\end{eqnarray*}}
\newcommand{\R}{\mathbb{R}}
\newcommand{\ds}{\displaystyle}

\def\argmin{\operatorname{argmin}}

\def\min{\operatorname{min}}
\def\max{\operatorname{max}}

\DeclareMathOperator{\grad}{grad}

\DeclareMathOperator{\diag}{diag}

\makeatletter
\renewcommand*{\@biblabel}[1]{\hfill#1.}
\makeatother

\begin{document}
\title{Iteration-complexity and asymptotic analysis  of steepest descent method for multiobjective optimization on Riemannian manifolds}

\author{
O. P. Ferreira \thanks{IME/UFG, Avenida Esperan\c{c}a, s/n, Campus Samambaia,  Goi\^ania, GO, 74690-900, Brazil (e-mails: {\tt orizon@ufg.br},  {\tt lfprudente@ufg.br}).}
\and
 M. S. Louzeiro\thanks{TU Chemnitz, Fakult\"at f\"ur Mathematik, D-09107,  Chemnitz, Germany (e-mail: {\tt mauricio.silva-louzeiro@mathematik.tu-chemnitz.de}).}
\and
L. F. Prudente \footnotemark[1]
}
\vspace{.5cm}
\maketitle

\noindent
\begin{abstract}
The steepest descent method for multiobjective optimization on Riemannian manifolds  with lower bounded sectional curvature is analyzed in this paper. The aim  of the paper is twofold. Firstly, an asymptotic analysis of the method is presented with three different finite procedures for determining the stepsize, namely, Lipschitz  stepsize,  adaptive stepsize  and  Armijo-type stepsize.  The second aim is to present, by assuming that the Jacobian of the objective function is componentwise Lipschitz continuous, iteration-complexity bounds for the method with these three  stepsizes strategies.   In addition, some examples are presented to emphasize the importance of working in this new context. Numerical experiments are provided to illustrate the effectiveness of the method in this new setting and certify the obtained theoretical results.
{\bf AMS } subject classification: 90C33, 49K05, 47J25. 
\end{abstract}
\section{Introduction}
A constrained  multiobjective optimization problem with  constraint  set  ${\cal M}$, consists of $m$ objective functions $f_1, \ldots, f_m$,  that have to be optimized at the same time on ${\cal M}$.  In recent years,  there has been a significant increase in the number of papers  addressing this class of  problems; for example, see \cite{PerezPrudente2018,gphz,Bento2018,Montonen2018,Carrizo2016,Fliege2016,Morovati2016}. Here, among the  methods  designed for solving multiobjective optimization problems, we are  interested in the  steepest descent method.  This method,  was proposed in  \cite{FliegeSvaiter2000} and since of  then  several variants  have been considered, including but  not limited to \cite{FukudaDrummond2013, FukudaDrummond2011, Drummond2005, Drummond2004, BelloCruzBouza2014,BelloCruz2013}. Recently  some  iteration-complexity results  to   gradient method for unconstrained  multi-objective optimization problem were presented in \cite{FliegeVazVicente2018}.   These results have been shown to be the same global rates as for  steepest descent method in scalar objective optimization.

Constrained optimization problems, where the constraint set  $\mathcal{M}$  can be endowed  with  Riemannian  manifold structure,    have  been studied  extensively in the last few years.  Some aspects about the use of Riemannian geometry tools to study these class of problems arises from the following interesting  fact. Endowing $\mathcal{M}$ with a suitable  Riemannian metric, an Euclidean non-convex constrained problem with  constraint  set  ${\cal M}$  can be seen as a Riemannian convex unconstrained  problem.  In addition to this property,  for differentiable functions,   its  gradient  can also become {\it Riemannian  Lipschitz continuous}; see \cite{FerreiraLouzeiroPrudente2018}.    Consequently,   the geometric and algebraic structures  that come  from the Riemannian metric make  possible to greatly reduce the computational cost for solving such  problems.  Indeed, it is well known that the iteration-complexity of several optimization methods for convex optimization problems such that  objective functions have  Lipschitz continuous  gradient  is much lower than    nonconvex  optimization  problems; see for example~\cite{BentoFerreiraMelo2017,JeurisVandebrilVandereycken2012,Rapcsak1997,SraHosseini2015, ZhangReddiSra2016} and references therein.   Furthermore, many   optimization problems are naturally posed on the Riemannian context;  see \cite{EdelmanAriasSmith1999,JeurisVandebrilVandereycken2012,  Smith1994, SraHosseini2015}. Then,  to take advantage of the  intrinsic Riemannian geometric structure, it is preferable to treat these  problems as the ones of  finding singularities of gradient vector fields  on  Riemannian manifolds rather than using Lagrange multipliers  or projection methods; see \cite{Luenberger1972,   Smith1994, UdristeLivro1994}.  In this sense constrained optimization problems can be seen as unconstrained from the point of view of Riemannian geometry. Moreover, intrinsic Riemannian structures can also opens up new research directions  that aid in developing competitive optimization algorithms;  see \cite{AbsilMahonySepulchre2008, EdelmanAriasSmith1999, JeurisVandebrilVandereycken2012, NesterovTodd2002, Smith1994, SraHosseini2015}.   More about   concepts and techniques of optimization on  Riemannian context can be found in  \cite{LiMordukhovichWang2011,LiYao2012,WangLiWangYao2015,WangLiYao2015,Manton2015,ZhangReddiSra2016,ZhangSra2016, UdristeLivro1994, WangLiLopezYao2016} and the bibliographies  therein.

In this paper we will study the steepest descent method for multiobjective optimization on Riemannian manifolds.  The aim  is twofold. First, asymptotic analysis will be done for quasi-convex and convex vectorial  functions.  In fact,   in \cite{BentoFerreiraOliveira2012}  asymptotic analysis of this method has already been done in Riemannian context; see also \cite{BentoNetoSantos2013}. However,  the analysis asymptotic  presented  in these previous works   is  just to stepsize given by  Armijo rule  and it  demand that the Riemannian manifolds have nonnegative sectional curvature.  The asymptotic analysis presented in the present  paper increase the previous ones in two different aspects.  {\it  It is provided an analysis   with three different finite procedures for determining the stepsize, namely, Lipschitz  stepsize,  adaptive stepsize  and  Armijo-type stepsize} and  only {\it lower boundedness  of the curvature} of  the Riemannian manifold is assumed.  The second aim is to present   {\it iteration-complexity bounds for steepest descent method for multiobjective optimization on Riemannian manifolds}. It is worth noting that, our results generalize  to the Riemannian context the results obtained in \cite{FliegeVazVicente2018}.  Besides, we present  one  iteration-complexity  bound   that is new even in Euclidean setting.  In addition, some examples are presented to emphasize the importance of working in this new context. Numerical experiments are provided to illustrate the effectiveness of the method in this new setting  and certify the  obtained theoretical results.

The organization of this paper is as follows. In Section~\ref{sec:auxmultiobj}, some notations and auxiliary results, used throughout of the paper, are placed. In Section~\ref{sec:GradientMult}, we present the algorithm  and the stepsizes that will be used. In Section~\ref{sec4}, the asymptotic convergence analysis of the sequence generated by the steepest descent method is made. In Section~\ref{Sec:IteCompAnalysis}, we present iteration-complexity bounds related to the steepest descent method. In Section~\ref{Sec:ExamplesRn}, we present  examples of vectorial convex functions with componentwise  Lipschitz continuous Jacobian. Numerical experiments are present in Section~\ref{numerical}. Finally, some conclusions are given in Section~\ref{conclusion}.

\section{Notations and  Auxiliary Concepts} \label{sec:auxmultiobj}
In this section, we recall  some  concepts, notations, and basics results  about Riemannian manifolds and vector optimization.   For more details we refer the reader to \cite{doCarmo1992, Sakai1996, UdristeLivro1994,  Rapcsak1997}. 

We denote by $T_p\mathcal{M}$ the {\it tangent space} of a finite dimensional  Riemannian manifold $\mathcal{M}$ at $p$, and by $T\mathcal{M}=\cup_{p\in M}T_p\mathcal{M}$ {\itshape{tangent bundle}} of ${\cal M}$.  The corresponding norm associated to the Riemannian metric $\langle \cdot ,  \cdot \rangle$ is denoted by $\|  \cdot \|$. We use $\ell(\alpha)$ to denote the length of a piecewise smooth curve $\alpha:[a,b]\to \mathcal{M}$. The Riemannian  distance  between $p$ and $q$   in  $\mathcal{M}$ is denoted  by $d(p,q)$.  Denote by ${\cal X}(\mathcal{M})$, the space of smooth vector fields on $\mathcal{M}$. Let $\nabla$ be the Levi-Civita connection associated to $(\mathcal{M}, \langle \cdot,\cdot \rangle)$.    For each $t \in [a,b]$ and a piecewise smooth curve $\alpha:[a,b]\to \mathcal{M}$, the covariant derivative $\nabla$ induces an isometry, relative to $ \langle \cdot , \cdot \rangle  $, $P_{\alpha,a,t} \colon T _{\alpha(a)} {\mathcal{M}} \to T_{\alpha(t)} {\mathcal{M}}$ defined by $ P_{\alpha,a,t}\, v = V(t)$, where $V$ is the unique vector field on the curve $\alpha$ such that $ \nabla_{\alpha'(t)}V(t) = 0$ and $V(a)=v$, the so-called {\it parallel transport} along  of  $\alpha$ joining  $\alpha(a)$ to $\alpha(t)$.  When there is no confusion,  $P_{\alpha,p,q}$ denotes  the parallel transport along  the  segment  $\alpha$ joining  $p$ to $q$.   Given that the geodesic equation $\nabla_{\ \gamma^{\prime}} \gamma^{\prime}=0$ is a second order nonlinear ordinary differential equation, then the geodesic $\gamma=\gamma _{v}( \cdot ,p)$ is determined by its position $p$ and velocity $v$ at $p$.  The restriction of a geodesic to a  closed bounded interval is called a {\it geodesic segment}.  For any two points $p,q\in\mathcal{M}$, $\Gamma_{pq}$ denotes the set of all geodesic segments  $\gamma:[0,1]\rightarrow\mathcal{M}$ with $\gamma(0)=p$ and $\gamma(1)=q$.  A geodesic segment joining $p$ to $q$ in $\mathcal{M}$ is said to be {\it minimal} if its length is equal to $d(p,q)$.      {\it In this paper, all manifolds are assumed to be  connected,   finite dimensional, and complete}. Hopf-Rinow's theorem asserts that any pair of points in a  complete Riemannian  manifold $\mathcal{M}$ can be joined by a (not necessarily unique) minimal geodesic segment.  Owing to  the completeness of the Riemannian manifold $\mathcal{M}$, the {\it exponential map} $\exp_{p}:T_{p}  \mathcal{M} \to \mathcal{M} $ is  given by $\exp_{p}v\,=\, \gamma _{v}(1,p)$, for each $p\in \mathcal{M}$.   For $f: {\cal M} \to\mathbb{R}$ a   differentiable function on $ \mathcal{M}$,   the Riemannian metric induces the mapping   $f\mapsto  \grad f $  which  associates     its {\it gradient} via the following  rule  $\langle \grad f(p), V(p)\rangle:= d f(p)V(p)$, for all $p\in {\cal M}$ and $V\in {\cal X}(\mathcal{M})$. For a twice-differentiable function,  the mapping  $f\mapsto \mbox{hess} f$ associates    its  {\it hessian} via the rule  $ \langle \mbox{hess} f V, V\rangle := d^2 f(V, V)$,  for all $V \in {\cal X}({\cal M})$, where    the last equalities imply that $ \mbox{hess} f V= \nabla_{V}  \grad f$,  for all $V \in {\cal X}({\cal M})$.  Let us to introduce some concepts of vector optimization on a Riemannian  manifold $\mathcal{M}$. Letting ${\cal I}:=\{1,\ldots,m\}$ define ${\mathbb R}^{m}_{+}:=\{x\in {\mathbb R}^m:~ x_{i}\geq 0,~~ i\in {\cal I}\}$ and ${\mathbb R}^{m}_{++}:=\{x\in {\mathbb R}^m: x_{i}> 0,~~ i\in {\cal I} \}$. For $x, \, y \in {\mathbb R}^{m}_{+}$, $y\succeq x$ (or $x \preceq y$) means that $y-x \in {\mathbb R}^{m}_{+}$ and $y\succ x$ (or $x \prec y$) means that $y-x \in {\mathbb R}^{m}_{++}$. Let   $F:=\left(f_1, \ldots, f_m\right):{\cal M} \to \mathbb{R}^m$ be a differentiable function. We denote the {\it Riemannian jacobian} of $F$ at a point $p\in \mathcal{M}$ by  $\nabla F(p)v:=\left(\langle \grad f_1(p), v \rangle, \ldots,\langle \grad f_m(p), v \rangle\right)$, where $v\in T_{p}\mathcal{M}$, and the image of the Riemannian jacobian of $F$ at $p$  by $\mbox{Im} (\nabla F(p)):=\left\{ \nabla F(p)v~: v\in T_pM \right\}.$  A   vectorial function $F:\mathcal{M} \to\mathbb{R}^m$ is said to be {\it convex} on $\mathcal{M}$ if  for any $p,q\in \mathcal{M}$ and $\gamma\in\Gamma_{pq}$  the composition $F\circ\gamma:[0, 1]\to\mathbb{R}$ satisfies $F\circ\gamma(t)\preceq (1-t)F(p)+tF(q),$ for all $t\in[0,1].$ By convexity of $F$, it follows that $\nabla F(p)\gamma'(0)\preceq F(q)-F(p)$. A vectorial function $F$ is called {\it quasi-convex} on $\mathcal{M}$ if, for every $p,q\in \mathcal{M}$ and  $\gamma\in\Gamma_{pq}$, it holds $F(\gamma(t))\preceq \max\{F(p),F(q)\}$,  for all $t\in [0,1]$, where the maximum is considered  coordinate by coordinate. It is immediate of the above definitions that if $F$ is convex then it is quasi-convex. Moreover, if $F$ is a quasi-convex function, than $F(q)\preceq F(p)$ implies $\nabla F(p)\gamma'(0)\preceq 0$. 

The next result plays  an important role in next sections.  Its proof, which will be omitted here,    follows  the same ideas as those presented in the proof of \cite[Lemma 3.2]{WangLiWangYao2015}, with some minor technical adjustments needed to settle it to our goals. {\it For simplifying  our  notations throughout the paper}, we define
\begin{equation} \label{eq:KappaHat}
\kappa<0, \qquad \hat{\kappa}:=\sqrt{|\kappa|}.
\end{equation}
\begin{lemma} \label{lem:comp}
Let $\mathcal{M}$ be a  complete Riemannian manifolds with sectional curvature $K\geq\kappa$. Let  $p,q\in \mathcal{M}$, $p\neq q$,  $v\in T_p\mathcal{M}$,   ${\gamma}:[0,\infty)\longrightarrow\mathcal{M}$ be  defined by ${\gamma}(t)=\mbox{exp} _{p}\left(tv\right)$ and  $\beta:[0,1]\rightarrow \mathcal{M}$ be a minimizing geodesic  with $\beta(0)=p$ and $\beta(1)=q$. Then, for any $t\in[0,\infty)$ there holds
\begin{multline*}
\cosh(\hat{\kappa}d(\gamma(t),q))\leq \cosh(\hat{\kappa}d(p,q))+\\
\hat{\kappa}\cosh(\hat{\kappa}d(p,q))\sinh(t\hat{\kappa}\left\|v\right\|)\left(\frac{t\left\|v\right\|}{2}-\frac{\tanh(\hat{\kappa}d(p,q))}{\hat{\kappa}d(p,q)} \frac{\left\langle v,\beta'(0)\right\rangle}{\left\|v\right\|}\right), 
\end{multline*}
and, consequently, the following inequality  holds 
$$
d^2({\gamma}(t),q)\leq d^2(p,q) + 
 \frac{\sinh\left(\hat{\kappa}t\|v\|\right)}{\hat{\kappa}} \left(t\|v\|\,\frac{\hat{\kappa}d(p,q)}{\tanh\left(\hat{\kappa}d(p,q)\right)}-\frac{2\left\langle v,\beta'(0)\right\rangle}{\left\|v\right\|}\right). 
$$
\end{lemma}
Next we present the definition of Lipschitz continuous gradient vector field; see \cite{da1998geodesic}.
\begin{definition} 
Let $f$ be a differentiable function on the set ${\cal M}$. The gradient vector  field    of $f$  is  said  to be   Lipschitz continuous on ${\cal M}$ with   constant $L\geq 0$ if,   for any $p,q\in{\cal M}$ and $\gamma\in\Gamma_{pq}$, it holds that 
$\left\|P_{{\gamma},p,q} \grad f(p)- \grad f(q)\right\|\leq L\ell(\gamma).$
\end{definition}
 The {\it norm of the hessian}   $\mbox{hess}\,f$   at $p\in{\mathcal M}$  is given by
$$ 
\|\mbox{hess}\,f(p)\|:= \sup  \left\{ \left\|\mbox{hess}\,f(p)v\right \|~:~  v\in T_{p}\mathcal{M}, ~\|v\|=1\right\}.
$$
The next result has similar proof to its Euclidean version and it will be omitted.
\begin{lemma} \label{le:CharactGL}
Let $f:\mathcal{M} \to \mathbb{R}$ be a twice continuously differentiable  function. The gradient vector  field   of $f$  is Lipschitz continuous with   constant $L\geq 0$ if,  and only if,  there exists $L\geq 0$ such that $\|\mbox{hess}\,f(p)\|\le L$, for all $p\in \mathcal{M}$.
\end{lemma}
In the following  we present the concept  of Lipschitz continuity for the Riemannian Jacobian  of a  vectorial  function.
\begin{definition} \label{Def:GradLips}
Let  $F:=\left(f_1, \ldots, f_m\right):{\cal M} \to \mathbb{R}^m$ be a differentiable function.  If for each  $f_i:{\cal M} \rightarrow \mathbb{R}$  there exists  a $L_i\geq 0$ such that  $\left\|P_{{\gamma},p,q} \grad f_i(p)- \grad f_i(q)\right\|\leq L_i\ell(\gamma)$,  for any $p,q\in{\cal M}$ and $\gamma\in\Gamma_{pq}$, then we say that   $\nabla F$ is componentwise  Lipschitz continuous on ${\cal M}$  with  constant  $L:= \max_{i=1,\ldots,m}\ L_i$.
\end{definition}
The proof of the next lemma follows, with  appropriate adjustments,  the same idea of  proof of the scalar version presented in   \cite[Corollary~2.1]{BentoFerreiraMelo2017}. {\it Throughout of the paper   we will use the following notation} 
$$
e:=(1,\ldots,1)\in \mathbb{R}^m.
$$
\begin{lemma} \label{le:lc}
Let  $F:=\left(f_1, \ldots, f_m\right):{\cal M} \to \mathbb{R}^m$ be a differentiable function.   Assume that  $\nabla F$  is  componentwise  Lipschitz continuous   on ${\cal M}$ with   constant $L\geq 0$ and $p\in {\cal M}$. Then   there holds 
\[
F( \exp_{p}(tv)) \preceq F(p) + t \nabla F(p)v + t^2\frac{L}{2}\left\|v\right\|^{2}e, \qquad \forall~t\in [0, +\infty),\quad v\in T_p\mathcal{M}.
\]
\end{lemma}
Next  we  introduce  the concept of quasi-Fej\'er  convergence, which played an important role in the analysis of the gradient method. 
\begin{definition} 
A sequence $\{y_k\}$ in the complete metric space $(\mathcal{M},d)$ is quasi-Fej\'er convergent to a set $W\subset \mathcal{M}$ if, for every $w\in W$, there exist  a sequence $\{\epsilon_k\}\subset\mathbb{R}$ such that $\epsilon_k\geq 0$, $\sum_{k=1}^{\infty}\epsilon_k<+\infty$, and $d^2(y_{k+1},w)\leq d^2(y_k,w)+\epsilon_k$, for all $k=0, 1, \ldots$.
\end{definition}
In the following we state the main  property of the  quasi-Fej\'er concept,   its proof follows the same path as its Euclidean counterpart proved in  \cite{burachik1995full},  by replacing the Euclidean distance by the Riemannian one. 
\begin{theorem}\label{teo.qf}
Let $\{y_k\}$ be a sequence in the complete metric space $(\mathcal{M},d)$. If $\{y_k\}$ is quasi-Fej\'er convergent to a nonempty set $W\subset \mathcal{M}$, then $\{y_k\}$ is bounded. Furthermore, if a cluster point $\bar{y}$ of $\{y_k\}$ belongs to $W$, then $\lim_{k\to\infty}y_k=\bar{y}$.
\end{theorem}

{\it Hereafter, we assume that $\mathcal{M}$  is a   complete Riemannian manifolds with sectional curvature $K\geq\kappa$,  where $\kappa<0$}. We point out that for Riemannian manifold with nonnegative sectional curvature, the convergence analysis  of the  steepest descent method  for convex and quasi-convex vector functions is well understood; see for example \cite{BentoFerreiraOliveira2012,BentoNetoSantos2013}.

\section{Steepest Descent for Multiobjective Optimization}\label{sec:GradientMult} 

Let   $F:=\left(f_1, \ldots, f_m\right):{\cal M} \to \mathbb{R}^m$ be a continuously differentiable function. The problem of finding an optimum Pareto  point of $F$,   we denote by
\begin{equation} \label{eq:opp}
\min \{ F(p) ~:~   p\in \mathcal{M}\}.
\end{equation}
A point $p\in \mathcal{M}$  satisfying  $\mbox{Im}(\nabla F(p))\cap (-{\mathbb R}^{m}_{++})=\emptyset$ is  called  {\it critical Pareto}.   An {\it optimum Pareto  point} of $F$ is a point $p_*\in \mathcal{M}$ such that there exists no other $p\in \mathcal{M}$ with $F(p)\preceq F(p_*)$ and $F(p) \neq F(p_*)$. Moreover, a point $p_*\in \mathcal{M}$ is a {\it weak  optimal Pareto} of $F$ if there is no $p\in \mathcal{M}$ with $F(p)\prec F(p_*)$. Consider the following problem 
\begin{equation} \label{eq:dsdd}
\mathop{\min}_{v\in T_{p}\mathcal{M}} \; \left\{ \max_{i\in {\cal I}} \left\langle\grad f_i(p),v \right\rangle+ \frac{1}{2}\|v\|^{2}\right\}, \quad \qquad {\cal I}=\{1,\ldots,m\}.
\end{equation}
Whenever  $p\in \mathcal{M}$  is not critical Pareto, the optimization problem \eqref{eq:dsdd} has only one solution, which is called  {\it  steepest descent direction} for $F$ in $p$ and it is denoted by 
\begin{equation}\label{def:descent.field}
v_p:=\argmin_{v\in T_{p}\mathcal{M}} \; \left\{ \max_{i\in {\cal I}} \left\langle\grad f_i(p),v \right\rangle+ \frac{1}{2}\|v\|^{2}\right\}.
\end{equation}
In the next lemma we state  an important  property of the  steepest descent direction. Its proof can be found in  \cite[Lemma 5.1]{BentoFerreiraOliveira2012}. 
\begin{lemma}\label{lem:convv1}
The steepest descent direction  mapping ${\cal M}\ni p\mapsto v_p \in T_pM$, is a continuous vector field.
\end{lemma}
Moreover, the vector $v_p$ is the solution of the problem  \eqref{eq:dsdd} if and only if there exist $\mu_j\geq 0$, for  $j\in {\cal I}(v_p):=\{j\in {\cal I}:~ \langle\grad f_j(p),v_p \rangle=\max_{i\in {\cal I}}\langle\grad f_i(p),v_p \rangle\}$,  such that 
\begin{equation} \label{eq:sddlc}
v_p=-\sum_{j\in {\cal I}(v_p)}\mu_j\grad f_j(p), \qquad \qquad  \sum_{j\in {\cal I}(v_p)}\mu_j=1, 
\end{equation}
see \cite[Lemma 4.1]{BentoFerreiraOliveira2012}.  In the following lemma we  state  an important inequality for our convergence analysis   and  an equivalence for a point $p\in {\cal M}$ to be a  critical Pareto.
\begin{lemma}\label{lem:ineq.aux}
Let  $p\in{\cal M}$ and $v_p$ as defined  \eqref{def:descent.field}. Then, 
\begin{equation} \label{eq:AuxIneq}
\max_{i\in { {\cal I}}}\langle\grad f_i(p),v_p \rangle= -\left\|v_p\right\|^2.
\end{equation}
Consequently,  $ \nabla F(p)v_p\preceq   -\left\|v_p\right\|^2e.$
In addition, $p$ is critical Pareto point  of $F$ if, and only if, $\left\|v_p\right\|=0$.
\end{lemma}
\begin{proof}
Let  $p\in{\cal M}$ and $v_p$ as in  \eqref{def:descent.field}. Thus,  from the first equality in   \eqref{eq:sddlc}  we have 
$$
-\left\|v_p\right\|^2= \left\langle-v_p, v_p \right\rangle= \left\langle  \sum_{j\in {\cal I}(v_p)}\mu_j\grad f_j(p), v_p\right\rangle= \sum_{j\in {\cal I}(v_p)}\mu_j\left\langle \grad f_j(p), v_p\right\rangle. 
$$
Hence, by the definition of ${\cal I}(v_p)$ and the  second  equality in   \eqref{eq:sddlc}, it is easy to verify that \eqref{eq:AuxIneq} holds. The second statement  follows  by using  the definitions of $\nabla F(p)v_p$ and ${\cal I}(v_p)$. We proceed with the prove of the third  statement of the lemma. Assuming that $p$ is a critical Pareto, it follows from the definition  that there exists $i\in {\cal I}$ such that $\left\langle \grad f_i(p),v_p\right\rangle \geq 0$. Then, the by first part of lemma we have $\left\|v_p\right\|=0$. The converse  follows from \cite[Lemma 4.2]{BentoFerreiraOliveira2012} and the proof  is concluded. 
\end{proof}
The proof of the next lemma is a straight combination of Lemma~\ref{le:lc} with  first part of Lemma~\ref{lem:ineq.aux} and will be omited.
\begin{lemma} \label{ineq:grad.lipschi}
 Assume that  $\nabla F$  is   componentwise  Lipschitz continuous   on ${\cal M}$ with   constant $L\geq 0$. Let $p\in\mathcal{M}$ and $v_p$ as defined in \eqref{def:descent.field}.  Then, there holds 
$$
F(\mbox{exp} _{p}\left(t\,v_p\right))\preceq F(p)+\left(\frac{Lt^2}{2}-t\right)\left\|v_p\right\|^2e,  \qquad \forall~  t\in [0,+\infty).
$$
\end{lemma}
Next we  state the steepest descent algorithm in  Riemannian manifold to solve \eqref{eq:opp}.\\

\begin{algorithm}[H]
\begin{description}
\item[ Step 0.] Let $ p_0\in{\cal M}$. Set $k=0$.
\item[ Step 1.]   Compute   $v_k:=v_{p_{k}}$, where  $v_{p_k}$ is defined in \eqref{def:descent.field}. If $v_{p_k}=0$, then {\bf stop}; otherwise,  choose a stepsize $t_k>0$ and compute
\begin{equation} \label{eq:GradMethod}
p_{k+1}:= \mbox{exp} _{p_{k}}\left(t_{k}\, v_{k}\right).
\end{equation}
\item[ Step 2.]  Set $k\leftarrow k+1$ and proceed to  \textbf{Step~1}.
\end{description}
\caption{Steepest descent algorithm in a Riemannian manifold $\mathcal{M}$}
\label{sec:gradient}
\end{algorithm}

Our goal is to analyze  Algorithm~\ref{sec:gradient} with three  different strategies for  choosing  the stepsize $t_k>0$. An analogous analysis done in the scalar case can be found  in \cite{FerreiraLouzeiroPrudente2018}. In the first  strategy we assume that  $\nabla F$  is componentwise  Lipschitz  continuous  and in  the last two without any Lipschitz condition. The  statements  of the strategies are as follows:
\begin{strategy}[Lipschitz stepsize]\label{fixed.step} 
Assume that  $\nabla F$  is   componentwise  Lipschitz continuous   on $\mathcal{M}$ with   constant $L\geq 0$. Let $\varepsilon>0$ and take
\begin{equation}\label{eq:fixed.step}
\varepsilon< t_k\leq  ~ \frac{1}{L}.
\end{equation}
\end{strategy}
Despite knowing that  $\nabla F$ is  componentwise  Lipschitz  continuous, in  general  the Lipschitz  constant is not  computable. Then, the next  strategy can be used to compute the stepsize   without any Lipschitz condition.  However, as we shall  show,  if  $\nabla F$ is componentwise  Lipschitz  continuous  with constant $L>0$ the   stepsize  computed is an approximation to  $1/L$; see the scalar case in  \cite{BeckTeboulle2009,FerreiraLouzeiroPrudente2018}. 
\begin{strategy}[adaptive stepsize]\label{adaptive.step}
Take $\zeta\in(0,1/2]$, $L_0>0$,  $t_0:= L_0^{-1}$, and $0<\eta<1$. Consider $v_k$ is defined as in \eqref{def:descent.field}. Set   $t_k:=   \eta^{i_k}t_{k-1}$, where 
\begin{equation} \label{eq:GradMethodcpl}
i_k:= \min \left\{ i:~F \left(\exp _{p_{k}}\left(\eta^{i}t_{k-1}v_k\right)\right)\preceq  F(p_{k})- \zeta {\eta^{i}}t_{k-1} \|v_k\|^2e,~ i=0,1, \ldots\right\}. 
\end{equation}
\end{strategy}
In the next  remark  we show  that if $\nabla F$  is   componentwise  Lipschitz continuous   on $\mathcal{M}$, the adaptive stepsize can be seen as  an approximation for $1/L$.
\begin{remark}
Suppose that  $\nabla F$  is   componentwise  Lipschitz continuous   on $\mathcal{M}$ with   constant $L> 0$.  Let  $L_{0}> 0$ be   an estimate for $L$ and  $v_k=v_{p_k}$ be  defined as in \eqref{def:descent.field}. Taking  $t=1/L$,  using  Lemma~\ref{ineq:grad.lipschi} and taking into account that $\zeta\leq 1/2$,  we obtain 
$$
F\left(\mbox{exp} _{p_{k}}\left(v_k/L\right)\right)\preceq F(p_k)- \left(\zeta\|v_k\|^2/L\right)e.
$$ 
Hence, it  follows that  $t_k=1/L$ is always accepted for Strategies~\ref{adaptive.step} with $i_k=0$. Therefore, if $ L_0\geq L$ then we have $t_k=  1/L_0$, i.e., the step-size is constant. On the other hand, if   $L_0\leq L$ then owing to $\eta<1$  we  conclude that    $t_k$ in Strategies~\ref{adaptive.step} satisfies 
\begin{equation}\label{des.arm.aux}
\frac{\eta}{ L}\leq t_k\leq \frac{1}{L_0}.
\end{equation}
\end{remark}

In the following strategy a stepsize satisfying an Armijo-type sufficient descent condition is chosen using a backtracking approach.

\begin{strategy}[Armijo-type stepsize]\label{armijo.step}
Let $t_{\max}>t_{\min}>0$, $0<\omega_1<\omega_2<1$ and $\delta\in(0,1)$. Let $v_k=v_{p_k}$ be defined as in \eqref{def:descent.field}. The stepsize  $t_k$ is  chosen  according the following  algorithm:
\begin{description}
\item[ {\sc Step 0}.] Set $\ell=0$ and take ${\hat t}_{k_0}\in [t_{\min}, t_{\max}]$.
\item[  {\sc Step 1}.]   If 
\begin{equation} \label{eq:sl}
 F \left(\exp _{p_{k}}({\hat t}_{k_\ell} v_k)\right)\preceq  F(p_{k})- \delta {\hat t}_{k_\ell} \|v_k\|^2e,
\end{equation}
then set    $t_k:={\hat t}_{k_\ell}$ and  {\bf stop}.
\item[  {\sc Step 2}.]  Choose a stepsize ${\hat t}_{k_{\ell+1}}\in [\omega_1{\hat t}_{k_\ell}, \omega_2 {\hat t}_{k_\ell}]$, set $\ell\leftarrow \ell+1$  and proceed to  {\sc Step~1}.
\end{description}
\end{strategy}
In the next remark we show that, for $\nabla F$     componentwise  Lipschitz continuous on $\mathcal{M}$,   the stepsizes  in  Strategy~\ref{armijo.step} are bounded below by a positive constant. 
\begin{remark}\label{steep.const.strat3}
Assume that  $\nabla F$  is   componentwise  Lipschitz continuous   on $\mathcal{M}$ with   constant $L\geq 0$,  $t_{\max}> 2[1-\delta]/L$ and $t_{\min}<2\omega_1(1-\delta)/L$.  Hence,  for any  $t\in(0, 2[1-\delta]/L]$,   from Lemma~\ref{ineq:grad.lipschi}  we have  
$$
F(\mbox{exp} _{p_{k}}\left(t\,v_k\right))\preceq F(p_k)-\delta t\|v_k\|^2e. 
$$
Therefore,  $t_k$ in Strategies~\ref{armijo.step} satisfies the inequality $t_k>t_{\min}$, for all  $k=0,1, \ldots.$
\end{remark}

Since  well-definedness  of  Strategies~\ref{adaptive.step} and ~\ref{armijo.step}    follows by using  ordinary arguments, we   will  omitted its proof here.  Hence, the sequence $\{p_k\}$ generated by Algorithm~\ref{sec:gradient} with  Strategies~\ref{fixed.step}, ~\ref{adaptive.step} or ~\ref{armijo.step} is well-defined.  Finally we remind  that,  $p$ is a critical Pareto if, and  only if,   $\left\|v_p\right\|=0$.  Therefore, {\it from now on   we assume that  $\left\|v_k\right\|\neq0$, for all $k$.  Moreover, let us denote by $\{p_k\}$  the infinity sequence    generated by  Algorithm~\ref{sec:gradient}}.

\subsection{Asymptotic Convergence Analysis}\label{sec4}
In this section, we analyze  asymptotic convergence of the sequence  $\{p_{k}\}$  generated by Algorithm~\ref{sec:gradient} with Strategies~\ref{fixed.step}, \ref{adaptive.step} and   \ref{armijo.step}. Let us define
$${\cal A}:=\{p\in M: F(p)\preceq F(p_{k}),\;\; k=0,1, \ldots\}.$$
To proceed with our analysis,       {\it from now on, we will assume that the set ${\cal A}$ is non-empty}. A condition guaranteeing this assumption is the existence of accumulation point for the  sequence $\{p_{k}\}$.

\begin{lemma}\label{lem:boud}
Let $\{p_k\}$ be    generated  with any of Strategies~\ref{fixed.step}, \ref{adaptive.step} or  \ref{armijo.step}. Then, 
	\begin{equation}\label{desi.des}
	F(p_{k+1})\preceq F(p_k)- \nu  t_k\left\|v_k\right\|^2e, \qquad k=0,1, \ldots,
	\end{equation}
where $\nu=1/2$ for Strategy~\emph{\ref{fixed.step}},  $\nu=\zeta$ for Strategy~\emph{\ref{adaptive.step}} and $\nu=\delta$ for Strategy~\emph{\ref{armijo.step}}. As a consequence, there holds 	 $\lim_{k\to +\infty}t_k\left\|v_k\right\|^2=~0$.
\end{lemma}
\begin{proof}
The   inequality \eqref{desi.des}  for  Strategies~\mbox{\ref{adaptive.step}} and {\ref{armijo.step}}  follows  from \eqref{eq:GradMethod},  \eqref{eq:GradMethodcpl} and \eqref{eq:sl}, respectively. Now, assume that  $\{p_k\}$ is   generated by using Strategies~\ref{fixed.step}. In  this case, combining    \eqref{eq:GradMethod} with Lemma~\ref{ineq:grad.lipschi} and taking into account that  \eqref{eq:fixed.step} implies  $(Lt_k/2-1)\leq-1/2 $, \eqref{desi.des} follows with $\nu=1/2$.  To proceed with the proof of  the last statement, take $q\in {\cal A}$ and an integer number  $\ell>0$. Thus,  \eqref{desi.des} yields 
$$
0\preceq \sum_{k=0}^{\ell}t_k\left\|v_k\right\|^2e \preceq \frac{1}{\nu} \sum_{k=0}^{\ell}\left(F(p_k)-F(p_{k+1})\right) \preceq \frac{1}{\nu} \left(F(p_0)-F(q)\right), 
$$
with implies  the desired result, and the proof of the lemma is concluded. 
\end{proof}
To simplify the statement and proof  of  the next  result we need to define  three  auxiliary constants.  For that, let $p_0\in \mathcal{M}$ . By using \eqref{desi.des} together with \eqref{eq:fixed.step}, \eqref{des.arm.aux} and \eqref{eq:sl} define the first  constant $ \rho >0$ as follows
\begin{equation} \label{eq:rho}
\sum_{k=0}^{\infty}t^2_k\left\|v_k\right\|^2\leq \rho :=
\begin{cases}
\min_{i\in {\cal I}}\left\{2[f_i(p_0)-f_i(q)]/L:~q\in {\cal A}\right\},   \, \, \, \, \, \,\,\, \mbox{ for  Strategy~\mbox{\ref{fixed.step}}};\\
\min_{i\in {\cal I}}\left\{[f_i(p_0)-f_i(q)]/(\zeta L_0):~q\in {\cal A}\right\},  \, \, \, \mbox{for Strategy~\mbox{\ref{adaptive.step}}};\\
\min_{i\in {\cal I}}\left\{t_{\max}[f_i(p_0)-f_i(q)]/\delta:~q\in {\cal A}\right\},   \,\, \,\mbox{for  Strategy~\mbox{\ref{armijo.step}}}.
\end{cases}
\end{equation}
The other  two auxiliaries  constants ${\cal C}_{\rho,\kappa}^q >0$ and  ${\cal K}_{\rho,\kappa}^q>0$ are defined as follows
\begin{align} 
{\cal C}_{\rho,\kappa}^q  &:=\cosh^{-1}\left(\cosh(\hat{\kappa}d(p_0,q))e^{\frac{1}{2}\left(\hat{\kappa}\sqrt{\rho}\right)\sinh\left(\hat{\kappa}\sqrt{\rho}\right)}\right), \label{eq:qkappa}\\
{\cal K}_{\rho,\kappa}^q  &:= \frac{\sinh\left(\hat{\kappa}\sqrt{\rho}\right)}{\hat{\kappa}\sqrt{\rho}}\frac{{\cal C}_{\rho,\kappa}^q }{\tanh {\cal C}_{\rho,\kappa}^q }, \label{eq:kkappa}
\end{align}
where the constants  ${\hat \kappa}$ and   $\rho$, are defined  in    \eqref{eq:KappaHat} and     \eqref{eq:rho}, respectively.
\begin{lemma}\label{lem:bounded}
Let $\{p_k\}$ be    generated  with any  of  Strategies~\ref{fixed.step}, \ref{adaptive.step} or  \ref{armijo.step} and  $q\in {\cal A}$. Assume that the function  $F$  is  quasi-convex on $\mathcal{M}$. Then,  
\begin{equation}\label{def:dqkapa}
d(p_{k+1},q)\leq \frac{1}{\hat{\kappa}} {\cal C}_{\rho,\kappa}^q,  \qquad  k=0, 1, \ldots. 
\end{equation}
As a consequence,  $\{p_k\}$ is bounded and   the following inequality  holds 
\begin{equation}\label{ineq:qf}
d^2(p_{k+1},q)\preceq d^2(p_k,q) + {\cal K}_{\rho,\kappa}^q  t_k^2\left\|v_k\right\|^2, \qquad \,\,~~~~k=0,1,\ldots.
\end{equation}
\end{lemma}
\begin{proof}
For each $k$, let   ${\gamma_k}:[0,\infty)\longrightarrow\mathbb{R}$ be  defined by ${\gamma_k}(t)=\mbox{exp} _{p_k}\left(tv_k\right)$.  Let  $\beta_k:[0,1]\rightarrow \mathcal{M}$ be a minimizing geodesic with $\beta_k(0)=p_k$ and $\beta_k(1)=q$. By using \eqref{eq:sddlc}, the definition of $v_k$, the  quasi-convexity of $F$, and taking into account that  $q\in {\cal A}$, we have
\begin{equation} \label{eq:bds}
\left\langle v_k, \beta'(0)\right\rangle=-\sum_{j\in {\cal I}(v_k)}\mu_j\left\langle \grad f_j(p_k), \beta'(0)\right\rangle\geq 0, \qquad \sum_{j\in {\cal I}(v_k)}\mu_j=1.
\end{equation}
Thus, applying  the first inequality of Lemma~\ref{lem:comp},  with $t=t_k$, $\gamma=\gamma_k$ , $\beta=\beta_k$ and   $p=p_k$, and  using \eqref{eq:GradMethod} and \eqref{eq:bds},   we obtain  
$$
\cosh(\hat{\kappa}d(p_{k+1},q))\leq \cosh(\hat{\kappa}d(p_k,q))\left(1+\frac{1}{2}\left( \hat{\kappa} t_k \left\|v_k\right\|\right)^2\frac{\sinh( \hat{\kappa} t_k\left\|v_k\right\|)}{\hat{\kappa}t_k\left\|v_k\right\|}\right).
$$
  Since  \eqref{eq:rho} implies  $ t_k \left\|v_k\right\|\leq \sqrt{\rho}$,  and   the map  $ (0, +\infty) \ni t \mapsto  \sinh(t)/t$ is increasing, we conclude that 
$$
\cosh(\hat{\kappa}d(p_{k+1},q))\leq \cosh(\hat{\kappa}d(p_k,q))\left(1+\sigma\left( t_k \left\|v_k\right\|\right)^2\right),
$$
where  $\sigma:= \hat{\kappa}(\sinh(\hat{\kappa}\sqrt{\rho}))/(2\sqrt{\rho})$.  Now note  that  the last inequality implies that 
$$
\cosh(\hat{\kappa}d(p_{k+1},q))\leq \cosh(\hat{\kappa}d(p_k,q))e^{\sigma\left( t_k \left\|v_k\right\|\right)^2}.
$$
Therefore, by   using \eqref{eq:rho},  it follows that  $\cosh (\hat{\kappa}d(p_{k+1},q))\leq\cosh(\hat{\kappa}d(p_0,q))e^{\sigma\rho}$ which,  considering the definition of $\sigma$ and \eqref{eq:qkappa}, yields \eqref{def:dqkapa}.  The boundedness of $\{p_k\}$ is immediate from \eqref{def:dqkapa}.   We proceed with the proof of \eqref{ineq:qf}. Now, we apply the  second inequality of  Lemma~\ref{lem:comp} and again we take into account \eqref{eq:GradMethod} and \eqref{eq:bds} to conclude that 
\begin{equation} \label{eq:GenIneq1}
d^2(p_{k+1},q)\leq d^2(p_k,q) +
 \frac{\sinh\left(\hat{\kappa}t_k\|v_k\|\right)}{\hat{\kappa}t_k\|v_k\|}\frac{\hat{\kappa}d(p_k,q)}{\tanh\left(\hat{\kappa}d(p_k,q)\right)} t_k^2\|v_k\|^2. 
\end{equation}
Since    the maps   $ (0, +\infty) \ni t \mapsto  t/\tanh(t) $ and $ (0, +\infty) \ni t \mapsto  \sinh(t)/t$  are  increasing and positive, taking into account    \eqref{def:dqkapa} and   that  $ t_k \left\|v_k\right\|\leq \sqrt{\rho}$, the inequality \eqref{eq:GenIneq1}   becomes  
$$
d^2(p_{k+1},q)\leq d^2(p_k,q) + \frac{\sinh\left(\hat{\kappa}\sqrt{\rho}\right)}{\hat{\kappa}\sqrt{\rho}}\frac{{\cal C}_{\rho,\kappa}^q }{\tanh {\cal C}_{\rho,\kappa}^q }t^2_k\|v_k\|^2.
$$
Therefore,  by  using \eqref{eq:kkappa}    we have the desired inequality.
\end{proof} 
In the next result we show  that  if $F$ is  a  quasi-convex function on a Riemannian manifolds with lower bounded sectional curvature, then $\{p_k\}$  converges  to  a critical Pareto point of $F$.

\begin{theorem}  \label{th:convthe}
Let $\{p_k\}$ be    generated  with any  of  Strategies~\ref{fixed.step}, \ref{adaptive.step} or  \ref{armijo.step}.  If  $F$ is quasi-convex, then  $\{p_k\}$ converges to a critical Pareto point of $F$.
\end{theorem}
\begin{proof}
Since ${\cal A}$ is non-empty,  Lemma~\ref{lem:bounded} and \eqref{eq:rho} imply that $\{p_k\}$ is bounded and quasi-Fej\'er convergent to set ${\cal A}$.  Taking into account Lemma~\ref{lem:boud}  we conclude  that  $\{ f_s(p_k)\}$  is non-increasing, for all $s=1, \ldots, m$.  Thus, we conclude that all cluster points  of $\{p_k\}$ belongs to ${\cal A}$. Hence,  Theorem~\ref{teo.qf} implies that  $\{p_k\}$ converges   to a point  $\bar{p}\in {\cal A}$. Hence, remais to prove that $\bar{p}$ is a critical Pareto point of $F$. We know  that, for  any of  the three  strategies~ ~\mbox{\ref{fixed.step}},  ~\mbox{\ref{adaptive.step}}  or \mbox{\ref{armijo.step}}, the sequence     $\{t_k\}$ is bounded. Let   $\bar{t}\geq 0$ be a  cluster  point of  $\{t_{k}\}$ and take $\{t_{k_j}\}$ such that  $\lim_{j\to\infty} t_{k_j}= \bar{t}$.  First  we suppose that   $\bar{t}>0$.   Since  $\lim_{j\to\infty} p_{k_j}= \bar{p}$ and $\lim_{j\to\infty} t_{k_j}= \bar{t}$,   \eqref{eq:rho} and  Lemma~\ref{lem:convv1}  imply that $0=\lim_{j\to\infty} t_{k_j}\left\|v_{k_j}\right\|= \bar{t}\left\| v_{\bar{p}}\right\|.$
Thus, considering that we are under the assumption $\bar{t}>0$,  we obtain    $v_{\bar{p}}=0$.  Therefore,  Lemma~\ref{lem:ineq.aux} implies that    $\bar{p}$ is a critical Pareto point of $F$.  Now, we suppose  that   $\bar{t}=0$.  In this case,   we  just need to  analyze  Strategies~\mbox{\ref{adaptive.step}}  and \mbox{\ref{armijo.step}},   due to  Strategy~\ref{fixed.step} we have $\epsilon\leq \bar{t}$. First assume that Strategy~\mbox{\ref{adaptive.step}}  is used and take $r\in\mathbb{N}$. Since  $\lim_{j\to\infty} t_{k_j}= 0$  we  conclude that if $j$ is large enough, $t_{k_j}<\eta^rt_{0}=:C_r$. Thus,  for each   $j$  large enough, from  \eqref{eq:GradMethodcpl} we have
$$
f_{s_j}\left(\exp_{p_{k_j}}(C_r v_{k_j})\right)> f_{s_j}(p_{k_j})- \zeta C_r  \|v_{k_j}\|^2,
$$
for some $s_j\in \{1,\ldots,m\}$. Since the set $\{1,\ldots,m\}$ is finite, without lose of generality, we assume the there exist ${\hat s}$ and a infinite set of index $j$  such that 
$$
f_{\hat s}\left(\exp_{p_{k_j}}(C_r v_{k_j})\right)> f_{\hat s}(p_{k_j})- \zeta C_r  \|v_{k_j}\|^2.
$$
Since  $\lim_{j\to\infty} p_{k_j}= \bar{p}$ and $\lim_{j\to\infty} t_{k_j}= \bar{t}$, letting $j$ goes to $+\infty$ and taking into account that $v_p$ and the exponential map  are  continuous, we obtain
$$
\frac{f_{\hat s}\left(\exp_{\bar{p}}(C_r v_{\bar{p}})\right)-f_{\hat s}(\bar{p})}{C_r}\geq -\zeta \|v_{\bar{p}}\|^2.
$$
Thus, letting $r$ goes to $+\infty$,  yields $\langle\grad f_{\hat s}(\bar{p}),v_{\bar{p}}\rangle \geq  -\zeta \|v_{\bar{p}}\|^2$. Hence,  from Lemma~\ref{lem:ineq.aux} we conclude that    $-\|v_{\bar{p}}\|^2\geq  -\zeta \|v_{\bar{p}}\|^2$ and,  considering that $\zeta\in (0,1/2]$, we have $\left\|v_{\bar{p}}\right\|=0$. Consequently, using  again Lemma~\ref{lem:ineq.aux} we have  $\bar{p}$ is a critical Pareto of $F$. Finally,  assume that  Strategy~\mbox{\ref{armijo.step}} is used.   Since  $\lim_{j\to\infty} t_{k_j}= 0$  we  conclude that if $j$ is large enough we have $t_{k_j}<t_{min}$. Thus,  if $j$ is large enough,  there exists $0<{\hat t}_{j}\leq t_{max}$ such that $ 0<\omega_1{\hat t}_{j}\leq{ t}_{k_j}$ and  
$$
f_{s_j}\left(\exp_{p_{k_j}}({\hat t}_{j} v_{k_j})\right)> f_{s_j}(p_{k_j})-{\hat t}_{j} \delta \|v_{k_j}\|^2,
$$
for some $s_j\in \{1,\ldots,m\}$.  Since the set $\{1,\ldots,m\}$ is finite, without lose of generality, we assume the there exist ${\hat s}$ and a infinite set of index $j$  such that 
$$
\frac{f_{\hat s}\left(\exp_{p_{k_j}}({\hat t}_{j} v_{k_j})\right)- f_{\hat s}(p_{k_j})}{{\hat t}_{j}}>- \delta \|v_{k_j}\|^2.
$$
Let $\gamma_j(t):= \exp_{p_{k_j}}( tv_{k_j})$, for $t>0$, be a geodesic segment. Thus, the mean value theorem implies that there exists   ${\bar t}_{j}\in(0,{\hat t}_{j})$ such that
\begin{equation}\label{ineq:non.arm.stra4}
\left\langle \grad f_{\hat s}\left(\gamma_j( {\bar t}_{j})\right),P_{\gamma_j,0,{\bar t}_{j}}v_{k_j}\right\rangle>- \delta \|v_{k_j}\|^2.
\end{equation}
  On the other hand, let $B_{\epsilon}(\overline{p})\subset\mathcal{M}$ be a totally normal  ball.  Hence, considering that  $\lim_{j\to +\infty} p_{k_j}=\bar{p}$,  Lemma~\ref{lem:convv1} implies that $\lim_{j\to +\infty} v_{k_j}=\bar{v}_{\bar{p}}$.   Moreover, $0<\omega_1{\hat t}_{j}\leq{ t}_{k_j}$ implies that $\lim_{j\to +\infty} \hat{t}_j=0$. Owing to  $0<{\bar t}_{j}\leq {\hat t}_{j}$  we obtain that  $\lim_{j\to +\infty} \bar{t}_j=0$. Hence, for all $j$ large enough we have $\{\bar{t}_j\}\subset (0,1)$ and $\gamma_j(\bar{t}_{j})\in B_{\epsilon}(\overline{p})$, which implies
$$
P_{\gamma_j,0,\bar{t}_{j}}v_{k_j} =\frac{1}{1-\bar{t}_{j}}\exp^{-1}_{\gamma_j(\bar{t}_{j})}\exp_{p_{k_j}}v_{k_j}.
$$
Thus, letting $j$ goes to $+\infty$ and using \cite[Lemma~1.1]{batista2018extragradient},  we conclude that $\lim_{j\rightarrow+\infty} P_{\gamma_j,0,\bar{t}_j}v_{k_j} =\bar{v}_{\bar{p}}$ (a general version for this equality, see \cite[Lemma~1.2]{batista2018extragradient}). Then, letting $j$ goes to $+\infty$ in \eqref{ineq:non.arm.stra4} and taking into account Lemma~\ref{lem:convv1}, that  $\grad f_{\hat s}$  and the exponential map  are  continuous, we obtain   $\langle\grad f_{\hat s}(\bar{p}),v_{\bar{p}}\rangle \geq  -\delta \|v_{\bar{p}}\|^2$. Hence, Lemma~\ref{lem:ineq.aux} implies that  $-\|v_{\bar{p}}\|^2\geq  -\delta \|v_{\bar{p}}\|^2$ and,  considering that $\delta\in (0,1)$, we have $\left\|v_{\bar{p}}\right\|=0$. Consequently, using again Lemma~\ref{lem:ineq.aux} we conclude that $\bar{p}$ is a critical Pareto of $F$.   Therefore,  for all Strategies~\ref{fixed.step}, \ref{adaptive.step} or  \ref{armijo.step},  $\bar{p}$ is a critical Pareto point of $F$, which concludes the proof.
\end{proof}

\begin{corollary}\label{res:fconv100}
Let $\{p_k\}$ be    generated  with any  of  Strategies~\ref{fixed.step}, \ref{adaptive.step} or  \ref{armijo.step}. If $F$ is convex, then  $\{p_k\}$ converges to a weak optimal Pareto  of $F$.
\end{corollary}
\begin{proof}
Since $F$ is convex,   critical points  are  weak optimal Pareto  of $F$, see \cite[Proposition~5.2]{BentoFerreiraOliveira2012}. Considering that convex functions are  also quasi-convex the result follows from Theorem~\ref{th:convthe}.
\end{proof}
\subsection{Iteration-Complexity Analysis} \label{Sec:IteCompAnalysis}
In this section we present  iteration-complexity bounds related to the steepest descent method with  Strategies~\ref{fixed.step}, ~\ref{adaptive.step} and ~\ref{armijo.step},  for  $F$ having  $\nabla F$ with  componentwise Lipschitz continuous  constant $L>0$. For this purpose, by using  \eqref{eq:fixed.step}, \eqref{des.arm.aux} and Remark~\ref{steep.const.strat3},  define
\begin{equation} \label{def:tmin}
{\xi}:=
\begin{cases}
\epsilon,    \,\qquad    \mbox{~ \;\;\; for  Strategy~\mbox{\ref{fixed.step}}};\\
\eta/ L,   \qquad   \mbox{~for Strategy~\mbox{\ref{adaptive.step}}};\\
t_{\min},    \qquad \mbox{~for  Strategy~\mbox{\ref{armijo.step}}}.
\end{cases}
\end{equation}
The following result extends   the scalar result   \cite[Theorem 3.1]{BentoFerreiraMelo2017} to multiobjective settings. Moreover, it also extends  to Riemannian context  \cite[Theorem 3.1]{FliegeVazVicente2018}. 
\begin{theorem}\label{th:grad1} 
Let $\{p_k\}$ be    generated  with any  of  Strategies~\ref{fixed.step}, \ref{adaptive.step} or  \ref{armijo.step}, and set $ f_{ i}^*:=\inf\{f_i(q):~  q\in {\cal M}\}$, for $i\in {\cal I}$.  Suppose that $f_{ i}^*$  is bounded from below for some $i\in  {\cal I}$, and define $i_*\in  {\cal I}$ such that
$$
 f_{i_*}(p_0)-f_{ i_*}^*:=\min \left\{ f_i(p_0)-f_{i}^*:~  i \in {\cal I} \right\}.
$$
 Then,   for every $N\in \mathbb{N}$,  there  holds
$$
\min \left\{\|v_{k}\|:~ k=0, 1,\ldots, N-1 \right\}\leq \left[\frac{ f_{i_*}(p_0)-f_{ i_*}^*}{\nu {\xi}}\right]^{\frac{1}{2}}\frac{1}{\sqrt{N}},
$$
where $\nu=1/2$ for Strategy~\emph{\ref{fixed.step}},  $\nu=\zeta$ for Strategy~\emph{\ref{adaptive.step}} and $\nu=\delta$ for Strategy~\emph{\ref{armijo.step}}. 
\end{theorem}
\begin{proof} 
It follows from Lemma~\ref{lem:boud} that $\nu  t_k\left\|v_k\right\|^{2}e\preceq F(p_k)-F( p_{k+1})$, for all $ k=0, 1, \ldots$. By summing both sides of this inequality  for $k=0, 1, \ldots, N-1$ and  using \eqref{def:tmin}, we obtain 
$$
\nu {\xi}\sum_{k=0}^{N-1}\left\|v_k\right\|^{2}e\preceq F(p_0)-F(p_{N}).
$$
Thus, by the definition of $i_*$, we conclude from the last inequality that 
 $$
 \nu {\xi} N\min\left\{\|v_{k}\|^2:~ k=0, 1,\ldots, N-1\right\}\leq  f_{i_*}(p_0)-f_{i_*}^*,
 $$
  which implies  the  statement of the theorem.  
\end{proof}
\begin{remark}
It is worth mentioning that in the above result it was not necessary to use any hypothesis  about  convexity of $F$  and curvature of ${\cal M}$.
\end{remark}
Now we are going to prove  that under the assumption of convexity Theorem~\ref{th:grad1} can be improved. We begin by presenting an auxiliary inequality. 
\begin{lemma}\label{pr:ltd}
Let $\{p_k\}$ be    generated  with any  of  Strategies~\ref{fixed.step}, \ref{adaptive.step} or  \ref{armijo.step}. Assume that  $F$ is a convex function on $\mathcal{M}$. Then, for  $q\in {\cal A}$ and  each $k$,  there   exist $\mu_{j's}^k\geq 0$ satisfying $\sum_{j\in {\cal I}(v_k)}\mu_j^k=1$ such that
\begin{equation}\label{eq;desgen}
d^2(p_{k+1},q)\leq d^2(p_k,q) +  {\cal K}_{\rho,\kappa}^qt_k^2\left\|v_k\right\|^2 + 2t_k\sum_{j\in {\cal I}(v_k)}\mu_j^k[f_j(q)-f_j(p_k)], 
\end{equation}
where $\rho$ is defined in \eqref{eq:rho}.
\end{lemma}
\begin{proof}
For each $k$, let  ${\gamma}_k:[0,\infty)\longrightarrow\mathbb{R}$ be  defined by ${\gamma_k}(t)=\mbox{exp} _{p_k}\left(tv_k\right)$ and $\beta_k:[0,1]\rightarrow \mathcal{M}$ with $\beta_k(0)=~p_k$ and $\beta_k(1)=q$ be  a minimizing geodesic.  Using \eqref{eq:sddlc} and the convexity of $F$  we conclude that exist $\mu_{j's}^k\geq 0$ satisfying $\sum_{j\in {\cal I}(v_k)}\mu_j^k=1$ such that
$$
\left\langle v_k, \beta_k'(0)\right\rangle=-\sum_{j\in {\cal I}(v_k)}\mu_j^k\left\langle \grad f_j(p_k), \beta_k'(0)\right\rangle \geq \sum_{j\in {\cal I}(v_k)}\mu_j^k(f_j(p_k)- f_j(q)).
$$
Applying the second inequality of  Lemma~\ref{lem:comp} with $\beta=\beta_k$,  $\gamma=\gamma_k$ and $t=t_k$ and using the last inequality  we obtain 
\begin{multline} \label{eq:GenIneq}
d^2(p_{k+1},q)\leq d^2(p_k,q) +\\
 \frac{\sinh\left(\hat{\kappa}t_k\|v_k\|\right)}{\hat{\kappa}t_k\|v_k\|}\left(\frac{\hat{\kappa}d(p_k,q)}{\tanh\left(\hat{\kappa}d(p_k,q)\right)}t_k^2\|v_k\|^2+
2t_k\sum_{j\in {\cal I}(v_k)}\mu_j^k(f_j(q)-f_j(p_k))\right). 
\end{multline}
Since   $ (0, +\infty) \ni t \mapsto  t/\tanh(t) $ and $ (0, +\infty) \ni t \mapsto  \psi(t):=\sinh(t)/t$  are  increasing, taking into account  that  \eqref{eq:rho} implies $ t_k \left\|v_k\right\|\leq \sqrt{\rho}$,  and using \eqref{def:dqkapa}, the inequality   \eqref{eq:GenIneq} becomes 
\begin{multline*}
d^2(p_{k+1},q)\leq d^2(p_k,q) +  \\
\frac{\sinh\left(\hat{\kappa}\sqrt{\rho}\right)}{\hat{\kappa}\sqrt{\rho}}\left(\frac{{\cal C}_{\rho,\kappa}^q }{\tanh {\cal C}_{\rho,\kappa}^q }t^2_k\|v_k\|^2+ 2t_k\sum_{j\in {\cal I}(v_k)}\mu_j^k(f_j(q)-f_j(p_k))\right). 
\end{multline*}
Therefore, due to $f_j(q)-f_j(p_k)\leq 0$  and $ \psi$ be bounded from below by $1$,  the inequality \eqref{eq;desgen}  follows by using \eqref{eq:kkappa},  which concludes the proof.
\end{proof}
The next result, with minor adjustments,  is a generalization of \cite[Theorem 4.1]{FliegeVazVicente2018}  to Riemannian setting,  when the Armijo's type strategy is used.
\begin{proposition}\label{prop:convcomb}
Let $\{p_k\}$ be    generated  with any  of  Strategies~\ref{fixed.step}, \ref{adaptive.step} or  \ref{armijo.step}. Assume that  $F$ is a convex function on $\mathcal{M}$ and $q\in {\cal A}$.  Then,   for every $N\in \mathbb{N}$,  there are non-negative numbers $\lambda_1,\ldots,\lambda_m$ with $\sum_{i=1}^{m}\lambda_i=1$, satisfying
\begin{equation}\label{teo:convcom}
\sum_{i=1}^m\lambda_i[f_i(p_N)-f_i(q)]\leq \frac{d^2(p_0,q) +  {\cal K}_{\rho,\kappa}^q \rho}{2{\xi} N}, 
\end{equation}
where $\rho$ is defined in \eqref{eq:rho}.
\end{proposition}
\begin{proof}
 Since $f_i(p_k)-f_i(q)\geq 0$ for all $i$,  Lemma~\ref{pr:ltd} and \eqref{def:tmin} implies  there   exist $\mu_{i's}^k\geq 0$ such  that 
$$
2{\xi}\sum_{i=1}^m\mu^k_i\left(f_i(p_k)-f_i(q)\right)   \leq  d^2(p_k,q)-d^2(p_{k+1},q) + {\cal K}_{\rho,\kappa}^qt_k^2\left\|v_k\right\|^2, 
$$
 and $\sum_{i=1}^m\mu_i^k=1$, where for each $k$, define $\mu_i^k:=0$ for all $i\notin{\cal I}(v_k)$.
By summing both sides of this inequality  for $k=0, 1, \ldots, N-1$, and using \eqref{eq:rho} follows
$$
2{\xi}\sum_{k=0}^{N-1}\sum_{i=1}^m\mu^k_i\left(f_i(p_k)-f_i(q)\right)  \preceq d^2(p_0,q) + {\cal K}_{\rho,\kappa}^q \rho.
$$
Since ${f_i(p_k)}$ is a decreasing sequence for each $i\in\{1,\ldots,m\}$, by some algebraic manipulations in the previous inequality we have
$$
\sum_{i=1}^m\left[\frac{1}{N}\sum_{k=0}^{N-1}\mu^k_i\right][f_i(p_N)-f_i(q)]   \leq \frac{d^2(p_0,q) + {\cal K}_{\rho,\kappa}^q \rho}{2{\xi} N}.
$$
Defining $\lambda_i:=\sum_{k=0}^{N-1}\mu^k_i/N$ we obtain the inequality in \eqref{teo:convcom}. To complete the proof, we have show that $\sum_{i=1}^m\lambda_i=1$. For that, it is sufficient  to  note that 
$$
\sum_{i=1}^m\lambda_i=\frac{1}{N}\sum_{i=1}^m\sum_{k=0}^{N-1}\mu^k_i=\frac{1}{N}\sum_{k=0}^{N-1}\sum_{i=1}^m\mu^k_i, 
$$
and    $\sum_{i=1}^m\mu_i^k=1$ for each $k$.
\end{proof}
Finally we are ready to   present the main result of this section, namely, the improvement of Theorem~\ref{th:grad1}. We remark that this result is new, even in Euclidean  context.
\begin{theorem}
Let $\{p_k\}$ be    generated  with any  of  Strategies~\ref{fixed.step}, \ref{adaptive.step} or  \ref{armijo.step}. Assume that  $F$ is a convex function on $\mathcal{M}$ and $q\in {\cal A}$.  Then,   for every $N\in \mathbb{N}$,  there  holds
$$
\min \left\{\|v_{k}\|:~ k=0, 1,\ldots, N \right\}\leq \left(\frac{2\left(d^2(p_0,q) +  {\cal K}_{\rho,\kappa}^q \rho\right)}{\nu {\xi}^2}\right)^{\frac{1}{2}}\frac{1}{N}.  
$$
where $\rho$ is defined in \eqref{eq:rho} and  $\nu=1/2$ for Strategy~\emph{\ref{fixed.step}},  $\nu=\beta$ for Strategy~\emph{\ref{adaptive.step}} and $\nu=\delta$ for Strategy~\emph{\ref{armijo.step}}. 
\end{theorem}
\begin{proof}
Let $N\in\mathbb{N}$ and denote by  ${\lceil N/2 \rceil }$ the least integer that is greater than or equal to $N/2$. It follows from Lemma~\ref{lem:boud} that $\nu  t_k\left\|v_k\right\|^{2}e\preceq F(p_k)-F( p_{k+1})$, for all $ k=0, 1, \ldots$.  Thus, by summing both sides of this inequality  for $k={\lceil N/2 \rceil },  \ldots, N$ and  using \eqref{def:tmin}, we obtain 
$$
\nu {\xi}\sum_{k={\lceil N/2 \rceil }}^{N}\left\|v_k\right\|^{2}\leq  f_i(p_{\lceil N/2 \rceil })-f_i(p_{N+1}), \qquad \forall ~i\in {\cal I} .
$$
Hence, taking non-negative numbers $\lambda_1,\ldots,\lambda_m$ as in the Proposition~\ref{prop:convcomb} and considering that  $q\in {\cal A}$,  we conclude  from the last inequality  that
$$
\nu {\xi}\sum_{k=\lceil N/2\rceil }^{N}\left\|v_k\right\|^{2} \leq \sum_{i=1}^m\lambda_i \left( f_i(p_{\lceil N/2\rceil })- f_i(p_{N+1})  \right)\leq \sum_{i=1}^m\lambda_i \left( f_i(p_{\lceil N/2\rceil })- f_i(q)  \right).
$$
Thus, from   Proposition~\ref{prop:convcomb}  and considering that $N/2\leq  \lceil N/2\rceil$ it follows that 
$$
\sum_{k=\lceil N/2\rceil }^{N}\left\|v_k\right\|^{2}\leq \frac{d^2(p_0,q) +  {\cal K}_{\rho,\kappa}^q \rho}{2\nu \xi^2 {\lceil N/2\rceil }}\leq \frac{d^2(p_0,q) +  {\cal K}_{\rho,\kappa}^q \rho}{\nu {\xi}^2N}.
$$
Therefore, $\min\{\|v_k\|^2:~ k=\lceil N/2 \rceil , \ldots, N\}\leq 2(d^2(p_0,q) +  {\cal K}_{\rho,\kappa}^q \rho)/(\nu {\xi}^2 N^2)$, which implies the  desired inequality.  
\end{proof}
\section{Examples}\label{Sec:ExamplesRn}
In this section we present some examples to illustrate the results obtained in previous sections. In particular, we  will  present some  examples  of  convex vectorial functions such that its  Riemannian Jacobian  is componentwise Lipschitz continuous.  
\begin{example} \label{ex:matrix1}
Let  ${\mathbb P}^n_{++}$ be the  cone of symmetric positive definite matrices. Define the vectorial function $F(X)=\left(f_1(X), \ldots, f_m(X)\right)$, where $f_i:{\mathbb P}^n_{++}\longrightarrow\mathbb{R}$ is given by 
\begin{equation}\label{eq:exmatrix}
f_i(X)= a_i\ln\left(\det(X)^{b_i} + c_i \right) - d_i\ln\left(\det(X)\right), 
\end{equation}
$a_i,b_i, c_i, d_i\in{\mathbb{R}_{++}}$ with $d_i<a_ib_i$ for all $i=1,\ldots,m$. Endowing ${\mathbb P}^n_{++}$ with the Riemannian metric given by  
$$
\langle U,V \rangle:= \mbox{tr} (VX^{-1}UX^{-1}),\qquad X\in {\mathbb P}^n_{++}, \qquad U,V\in
T_X{\mathbb P}^n_{++},
$$
where $\mbox{tr}(X)$ denotes the trace of  $X\in {\mathbb P}^n$, we obtain a Riemannian manifolds $\mathcal{M}:= ({\mathbb P}^n_{++}, \langle \cdot , \cdot \rangle)$with nonpositive sectional curvature, see \cite[Theorem 1.2. p. 325]{Lang1999}. In ${\cal M}$,  $f_i$ is convex and has Lipschitz gradient with constant $L_{i}\leq a_ib_i^2n$,  for each $i=1,\ldots,m$, see \cite[example 4.5]{FerreiraLouzeiroPrudente2018}. Hence,  from Definition~\ref{Def:GradLips} the Jacobian  $\nabla F$ is componentwise  Lipschitz continuous with constant $L\leq n\max\{a_1b_1^2, \ldots, a_mb_m^2 \}$.  In  $\mathcal{M}$,  the {\it exponential mapping   } $\exp_X:T_X\mathcal{M}\to \mathcal{M}$,  is given by
\begin{equation}\label{eq:expM}
\exp_X(V)=X^{1/2}e^{X^{-1/2}VX^{-1/2}}X^{1/2}, \qquad \quad V\in {\mathbb P}^n,\quad  X\in {\mathbb P}^n_{++}.
\end{equation} 
Therefore, from Corollary~\ref{res:fconv100} we can apply Algorithm~\ref{sec:gradient} with Strategies~\ref{fixed.step}, \ref{adaptive.step} or  \ref{armijo.step} to find  weak  optimal  Pareto  of $F$. 
\end{example}

In the following we present, without giving the details, one more example of convex vectorial function with Lipschitz gradients in the Riemannian manifolds $\mathcal{M}:= ({\mathbb P}^n_{++}, \langle \cdot , \cdot \rangle)$.
\begin{example} \label{ex:matrix2}
Let  $F(X)=\left(f_1(X), \ldots, f_m(X)\right)$ be a vectorial function, where  $f_i:{\mathbb P}^n_{++}\to\mathbb{R}$ is defined by 
$$
f_i(X)= a_i\ln(\det(X))^2-b_i\ln\left(\det(X)\right), 
$$
$a_i,b_i\in{\mathbb{R}_{++}}$ for all $i=1,\ldots,m$.  In $\mathcal{M}:= ({\mathbb P}^n_{++}, \langle \cdot , \cdot \rangle)$,  $f_i$ is convex and has Lipschitz gradient with constant $L_{i}\leq 2a_i\sqrt{n}$,  for each $i=1,\ldots,m$, \cite[example 4.4]{FerreiraLouzeiroPrudente2018}. The Jacobian  $\nabla F$ is componentwise  Lipschitz continuous with constant $L\leq 2\sqrt{n}\max\{a_1, \ldots, a_m \}$. 
\end{example}
Now,   we present  some preliminaries results to study examples of  convex vectorial functions with componentwise Lipschitz continuous Riemannian Jacobians.   We begin  with a result that, with some adjustments  in the notation, can be found in  \cite[Lemma~2]{LinHeZhang2013}.
\begin{lemma}\label{lem:iso.cam.par}
Let  ${\bar{\mathcal M}}$ and ${\mathcal M}$ be Riemannian manifold,  ${\bar{\nabla}}$ be the Levi-Civita connection associated to ${\bar{\mathcal M}}$ and $\varphi: {\bar{\mathcal M}}\rightarrow {\mathcal M}$ be  an isometry. Then,    ${\nabla}:{\cal X}(\mathcal{M})\times{\cal X}(\mathcal{M})\rightarrow {\cal X}(\mathcal{M})$ defined by
\begin{equation} \label{eq:NewConection}
\nabla_{V}U:= \mbox{d}\varphi(\bar{\nabla}_{\bar V}{\bar U}), \qquad \,\forall~ { V}, {U}\in {\cal X}(\mathcal{M}).
\end{equation}
is the   Levi-Civita connection associated to ${\bar{\mathcal M}}$, where   ${\bar V}=\mbox{d}\varphi^{-1}{V}$ and ${\bar U}=\mbox{d}\varphi^{-1}{U}$.
\end{lemma}
\begin{proof}
Let $f$ be   continuously differentiable, ${V}$ and  ${U}$ be vector fields in ${\mathcal M}$. 
Since $\varphi$ is a diffeomorphism,  ${f}\circ \varphi$ is  continuously differentiable,   ${\bar V}=\mbox{d}\varphi^{-1}{V}$ and ${\bar U}=\mbox{d}\varphi^{-1}{U}$  are vector fields in ${\bar{\mathcal M}}$. Thus, we can prove that \eqref{eq:NewConection} satisfies  \cite[equations (1.9), (1.10), (1.11) and (1.12) on page 27 and 28]{Sakai1996}  and therefore is the   Levi-Civita connection associated to $\mathcal{M}$.
\end{proof}
The next result is the main tool used in the following examples. 
\begin{theorem}\label{The:iso.conv.fun.a}
Let ${\mathcal M}$ and ${\bar{\mathcal M}}$ be Riemannian manifolds, $f:{\mathcal M}\rightarrow {\mathbb R}$ be a twice-differentiable function and $\varphi:{\bar{\mathcal M}}\rightarrow {\mathcal M}$  be an isometry. Then,   $f$ has gradient vector field Lipschitz continuous with constant $L\geq 0$ if, and only if,   $g:{\bar{\mathcal M}}\rightarrow {\mathbb R}$ defined by $g :=f\circ\varphi$, has gradient vector field Lipschitz continuous with constant $L\geq 0$.
\end{theorem}
\begin{proof}
Let  ${\bar V}\in {\cal X}(\bar {\mathcal{M}})$ and set  ${V}(\varphi(q))=\mbox{d}\varphi(q) {\bar V}(q)$. Thus,   by using   the definition of the gradient vector field and the chain rule, we have
\begin{align*}
\left\langle \grad g(q),{\bar V}(q)\right\rangle&=\mbox{d}g(q){\bar V}(q)\\
                                                             &=\mbox{d}f (\varphi(q))\mbox{d}\varphi(q){\bar V}(q)\\
                                                             &=\mbox{d}f (\varphi(q)) {V}(\varphi(q)) \\
                                                             &=\left\langle \grad f (\varphi(q)), {V}(\varphi(q))\right\rangle.
\end{align*}
Taking into account  that  $\varphi$ is an isometry and  ${V}(\varphi(q))=\mbox{d}\varphi(q) {\bar V}(q)$, we obtain that 
\begin{align*}
\left\langle  \grad f (\varphi(q)),  {V}(\varphi(q))\right\rangle&= \left\langle  \grad f (\varphi(q)),  \mbox{d}\varphi(q) {\bar V}(q)\right\rangle\\
                                                                                                    &= \left\langle   \mbox{d}\varphi(q) \mbox{d}\varphi(q)^{-1}  \grad f (\varphi(q)),  \mbox{d}\varphi(q) {\bar V}(q)\right\rangle\\
                                                                                                    &= \left\langle  \mbox{d}\varphi(q)^{-1}  \grad f (\varphi(p)),   {\bar V}(q)\right\rangle.
\end{align*}
Hence, combining the two above equality   we conclude that  $\grad f (\varphi(q))=\mbox{d}\varphi(q) \grad g(q)$. Moreover,  the definition of the hessian of $f$ together with Lemma~\ref{lem:iso.cam.par} yield
\begin{align*}
\mbox{hess}\, f(\varphi(q))\mbox{d}\varphi(q) {\bar V}(q)&=\mbox{hess}\, f(\varphi(q)){V}((\varphi(q))\\
                                                                         &=\nabla_{{V}(\varphi(q))}\grad f(\varphi(q))\\
                                                                         &= \mbox{d}\varphi(p)\left( {\bar{\nabla}}_{{\bar V}(q)}\grad g(q)\right)\\
                                                                         &= \mbox{d}\varphi(q) \mbox{hess} \,g(q){\bar V}(q), 
\end{align*}
which implies that   $\mbox{hess}\, f(\varphi(q))\mbox{d}\varphi(q) = \mbox{d}\varphi(q) \mbox{hess} \,g(q)$. Then, using again that $\varphi$ is an isometry,  we have $\|\mbox{hess}\, f(\varphi(q))\|=\|\mbox{hess}\,g (q)\|.$ Therefore, by using   Lemma~\ref{le:CharactGL} the results follows.
\end{proof}
The next result is an important property of isometries, its prove is in \cite[Proposition 5.6.1, p. 196]{Petersen2016}.
\begin{proposition}\label{prop:geod.iso} 
Let ${\mathcal M}$ and ${\bar{\mathcal M}}$ be  complete Riemannian manifolds. If $\varphi:{\bar{\mathcal M}} \rightarrow {\mathcal M}$ is a isometry and $\gamma$ is a geodesic in ${\bar{\mathcal M}}$, then $\varphi\circ\gamma$ is a geodesic in ${\mathcal M}$.
\end{proposition}

The following result is a straight consequence of the definition of isometry and Proposition~\ref{prop:geod.iso}.
\begin{theorem}\label{prop:iso.conv.fun}
Let ${\mathcal M}$, ${\bar{\mathcal M}}$ be Riemannian manifold and $\varphi:{\bar{\mathcal M}} \rightarrow {\mathcal M}$  an isometry. The function $g:{\mathcal M}\rightarrow {\mathbb R}$  is convex if and only if $f:{\bar{\mathcal M}} \rightarrow {\mathbb R}$, defined by
$f(p)=(g\circ\varphi)(p)$, is convex.
\end{theorem}

In the next example we change  the metric of the Euclidean space $\mathbb{R}^n$ to prove, in particular,  that  the  {\it extended  Rosenbrock's banana function} is convex and has gradiente Lipschitz in $\mathbb{R}^n$ with this new metric.    It is worth to pointed out that the  convexity of this function in two dimension  has been established in \cite[p. 83]{UdristeLivro1994}. 
\begin{example}[Rosenbrock's banana function class]\label{ex:fun.rosen.r2}
Let  $f_j: \mathbb{R}^{2n} \to \mathbb{R}$ be  a variant of the Rosenbrock's banana function, defined by
\begin{equation}\label{eq:ros}
 f_j(x_1,\ldots, x_{2n}):=\sum_{i=1}^{n}a_{ij}\left(x_{2i-1}^2-x_{2i}\right)^2 + \left(x_{2i-1}-b_{ij}\right)^2,  \quad a_{ij}\in\mathbb{R}_{++}, \quad b_{ij}\in\mathbb{R},
\end{equation}
for $j=1, \ldots, m$. Denote    ${\bar{\mathcal M}}$ as   the Euclidean space $\mathbb{R}^{2n}$ with the usual metric.  It is  well known  that  $f_j$  is non-convex and its  gradient  is non-Lipschitz  continuous  in ${\bar {\mathcal M}}$. Endowing $\mathbb{R}^{2n}$ with the  new Riemannian metric $\langle u,v \rangle:=u^TG(x)v$, where $u, v\in \mathbb{R}^{2n}$ and $G(x)$ is the $2n\times 2n$ block diagonal matrix $G(x)=\diag(G_{1}(x),\ldots, G_{n}(x))$, where the  blocks  are given by
\begin{equation*}
G_i(x):=
\begin{pmatrix}
 1+4x_{2i-1}^2 & -2x_{2i-1} \\
 -2x_{2i-1} & 1
\end{pmatrix},\qquad i=1,\ldots,n,
\end{equation*}
and $x=(x_1,\ldots, x_{2n})$,  we obtain a Riemannian manifold ${\mathcal M}:= (\mathbb{R}^{2n},G)$. Taking into account that the function $\varphi: {\bar{\mathcal M}} \to {\mathcal M}$ defined by 
$$
\varphi(z_1,\ldots,z_{2n})= \left(z_1,z_1^2-z_2, \ldots, z_{2n-1},z_{2n-1}^2-z_{2n}\right), 
$$
 is an isometry,  the Riemannian manifolds  ${\mathcal M}$ is complete and has constant seccional curvature $K=0$.    On the other hand,  $g_j:  {\bar{\mathcal M}} \to \mathbb{R}$ defined by 
\begin{equation*}
g_j(z_1,\ldots,z_{2n}):=(f_j\circ\varphi)(z_1,\ldots,z_{2n})= \sum_{i=1}^{n}a_{ij}z_{2i}^2+(z_{2i-1}-b_{ij})^2,  \qquad  j=1, \ldots, m, 
\end{equation*}
is a quadratics function, which is convex with  gradient vector field Lipschitz in $ {\bar{\mathcal M}}$ with constant $L_j:=\max\{2,2a_{1j},\ldots,2a_{nj}\}$. Therefore,  Theorem~\ref{prop:iso.conv.fun} and  Theorem~\ref{The:iso.conv.fun.a}  imply, respectively,  that $f_j$ is also convex and   has  gradient vector field Lipschitz continuous, with constant $L_j$, in ${\mathcal M}$.    Let $F=\left(f_1, \ldots, f_m\right)$ be the Rosenbrock's banana  vectorial function. Hence,  $F$ is convex and Definition~\ref{Def:GradLips} implies that $\nabla F$ is componentwise  Lipschitz continuous with constant $L=\max\{2, 2a_{11}, \ldots 2a_{nm}\}$.  The gradient of $f_j$ is given by  $\grad f_j(x)=G(x)^{-1}f'_j(x)$, where $f'_j$ is the usual gradient of $f_j$.  Given  $z\in{\bar{\mathcal M}}$  the exponential map in  ${\bar{\mathcal M}}$,   ${\overline \exp}_z:T_z{\bar{\mathcal M}}\rightarrow{\bar{\mathcal M}}$,  is given by ${\overline \exp}_z({\bar v})=z+{\bar v}$.  Since $\varphi$ is an isometry,  Proposition~\ref{prop:geod.iso} implies  that  the exponential map in  ${\mathcal M}$, $\exp_x:T_x{\mathcal M}\to {\mathcal M}$,  is given by
$
\exp_{x}(v)=\varphi(\varphi^{-1}(x)+{\mbox{d}\varphi}^{-1}(x)v).
$
Thus,  due to   $\varphi^{-1}(x)=(x_1,x_1^2-x_2, \ldots,  x_{2n-1},x_{2n-1}^2-x_{2n})$ and   ${\mbox{d}\varphi}^{-1}(x)v= (v_1,2x_1v_1-v_2, \ldots,  v_{2n-1},2x_{2n-1}v_{2n-1}-~v_{2n})$, we obtain that 
$$
\exp_{x}(v)= \left(x_1+v_1,v_1^2+x_2+v_2,  \dots,p_{2n-1}+v_{2n-1},v_{2n-1}^2+x_{2n}+v_{2n} \right)\,
$$
where $x:= (x_1,\ldots,  x_{2n})$ and  $v:= (v_1,\ldots,  v_{2n})$.  
\end{example}

We end this section by presenting, in particular, a family of vectorial functions in positive orthant $\mathbb{R}_{+}^n$  that are not convex and their gradients are not componentwise Lipschitz continuous. However,    by a suitable  change of the metric of $\mathbb{R}_{+}^n$ the functions of that family   are convex and  have componentwise Lipschitz continuous gradients on this new Riemannian manifold.

\begin{example}\label{ex:fun.product}
Let $f_j: \mathbb{R}^n_{++} \to \mathbb{R}$ be defined by
\begin{equation} \label{eq:psy}
  f_j(x):=a_j\ln\left(\prod_{i=1}^{n}x_{i}^{u_{ij}}+b_i\right) -  \sum_{i=1}^{n}w_{ij}\ln(x_i) +  c_j \sum_{i=1}^{n}\ln^2(x_i), 
\end{equation}
where $x:=(x_1,\ldots,  x_n)\in  \mathbb{R}^n_{++}$, $u_j:=(u_{1j}, \ldots, u_{nj})^T\in\mathbb{R}_{+}^n$, $w_j:=(w_{1j}, \ldots, w_{nj})^T\in\mathbb{R}_{+}^n$ and $a_j, b_j, c_j\in\mathbb{R}_{++}$, for all $j=1, \ldots,m$. Denote    ${\bar{\mathcal M}}$ as   the Euclidean space $\mathbb{R}^n$ with the usual metric.  The   function $f$  is in general non-convex and its  gradient  is non-Lipschitz  in ${\bar {\mathcal M}}$. Endowing $\mathbb{R}^n_{++}$ with the new Riemannian metric   $\langle u,v \rangle:=u^TG(x)v$, where $u, v\in T_x{\mathcal M}$  and   $G(x)$ is  the  $n\times n$  diagonal matrix 
$$
G(x):=\diag\left(x_{1}^{-2},x_{2}^{-2},\dots,x_{n}^{-2}\right), 
$$
we obtain the Riemannian manifold ${\mathcal M}:=(\mathbb{R}_{++}^n,G)$.  Since $\varphi: {\bar{\mathcal M}} \to {\mathcal M}$ defined by 
\begin{equation} \label{eq:isrnp}
\varphi(z_1,\dots,z_n)=\left(e^{z_1},\dots ,  e^{z_{n}}\right), 
\end{equation}
 is an isometry,  then   ${\mathcal M}$ is complete and has constant seccional curvature $K=0$.   The function   $g_j:  {\bar{\mathcal M}} \to \mathbb{R}$ defined by 
\begin{equation*}
g_j(z):=(f_j\circ\varphi)(z)=  a_j\ln\left(e^{u_j^Tz}+b_j\right) - w_j^Tz +c_jz^Tz, \qquad \quad  z:=(z_1,\ldots, z_n)^T\in  {\bar{\mathcal M}} , 
\end{equation*}
is convex and its  gradient  is Lipschitz in $ {\bar{\mathcal M}}$ with constant $L_j\leq a_ju_j^Tu_j/b_j + 2c_j$.  Thus,  Theorem~\ref{prop:iso.conv.fun} and  Theorem~\ref{The:iso.conv.fun.a}  imply, respectively,  that $f_j$ is also convex and   has  gradient Lipschitz  in ${\mathcal M}$  with constant $L_j$.  Therefore,  the  vectorial function $F(x)=\left(f_1(x), \ldots, f_m(x)\right)$ is convex and Definition~\ref{Def:GradLips} implies that $\nabla F$ is componentwise  Lipschitz continuous with constant $L=\max\{L_{1}, \ldots, L_{m} \}$.  The    gradient  of $f_j$ is given by 
$$\grad f_j(x)= \diag (x)^{2}f_j'(x),    \qquad x\in{\mathcal M}$$
where $\diag (x):= \diag(x_{1},\ldots,x_{n})$ and $f_j'$ is the usual derivative.  Using the isometry \eqref{eq:isrnp}   Proposition~\ref{prop:geod.iso} implies that  the exponential map in  ${\mathcal M}$, $\exp_x:T_x{\mathcal M}\to {\mathcal M}$,  is given by
$
\exp_{x}(v)=\varphi(\varphi^{-1}(x)+{\mbox{d}\varphi}^{-1}(x)v).
$
Since $\varphi^{-1}(x)=\left(\ln{x_1},\dots ,  \ln{x_{n}}\right)$ and ${\mbox{d}\varphi}^{-1}(x)v= (x_1^{-1}v_1,\ldots, x_n^{-1}v_n)$,  where $v= (v_1,\ldots,  v_n)$, we have 
$$\exp_x(v) = \left(x_1e^{\frac{v_1}{x_1}},\ldots,x_ne^{\frac{v_n}{x_n}}\right), \qquad v:=(v_1,\ldots,v_n)\in T_{x}\mathcal{M}\equiv\mathbb{R}^n.$$
\end{example}

\section{Numerical experiments}\label{numerical}

In order to illustrate the applicability of our proposal, we implemented Algorithm~\ref{sec:gradient} with the Armijo-type stepsize and tested it in the functions of the examples in Section~\ref{Sec:ExamplesRn}. Without attempting to go into details, we mention that the Armijo-type line search sketched out in Strategy~\ref{armijo.step} was coded based on (quadratic) polynomial interpolations of the coordinate functions.  We refer the reader to \cite{lppwolfe} for a careful discussion about line search strategies for vector optimization problems. We set $\delta = 10^{-4}$, $t_{\min} = 10^{-2}$, $t_{\max} = 10^{2}$, $\omega_1 = 0.05$, and $\omega_2 = 0.95$. Given a Riemannian manifold $\mathcal{M}$, the steepest descent direction $v_p$ at a non-critical point $p\in \mathcal{M}$ as in \eqref{def:descent.field} can be calculated by solving for $\lambda \in \R$ and $u\in T_p\mathcal{M}$ the following differentiable problem
\begin{equation} \label{vproblem}
 \begin{array}{ll}
\mbox{Minimize}   & \ds \lambda + \frac{1}{2}\left\langle u,u \right\rangle         \\
\mbox{subject to} & \left\langle \grad f_i(p),u \right\rangle \leq \lambda, \quad i=1,\ldots,m,
\end{array}
\end{equation}
which is a convex quadratic problem with linear inequality constraints, see \cite{FliegeSvaiter2000}. In our implementation, for calculating $v_p$, we solve problem \eqref{vproblem} using Algencan~\cite{algencan}, an augmented Lagrangian code for general nonlinear programming.

We stopped the execution of the algorithm at $p_k$ declaring convergence if
$$\max_{i\in {\cal I}} \left\langle\grad f_i(p_k),v_k \right\rangle+ \frac{1}{2}\|v_k\|^{2} \geq - 5 \times \texttt{eps}^{1/2},$$
where ${\cal I}=\{1,\ldots,m\}$, and $\texttt{eps}$ denotes the machine precision given. In our experiments we used $\texttt{eps}=2^{-52}\approx 2.22 \times 10^{-16}$. We point out that this convergence criterion was proposed in the numerical tests of~\cite{mauriciobenar&fliege} and also used in \cite{FerreiraLouzeiroPrudente2018,PerezPrudente2018}. The maximum number of allowed iterations was set to 10000. Codes are written in double precision Fortran~90 and are freely available at \url{https://orizon.ime.ufg.br/}.

\subsection{Rosenbrock's Problem}

We start the numerical experiments by verifying the practical behavior of Algorithm~\ref{sec:gradient} in a small instance of the Rosenbrock's problem given by the functions in Example~\ref{ex:fun.rosen.r2}. We considered $n=1$, $m=2$ in \eqref{eq:ros}, and set  $F(x)=(f_1(x),f_2(x))$ where
\begin{equation}\label{eq:ros1}
f_1(x_1,x_2)=100\left(x_{1}^2-x_{2}\right)^2 + \left(x_{1}-1\right)^2,
\end{equation}
\begin{equation}\label{eq:ros2}
f_2(x_1,x_2)=100\left(x_{1}^2-x_{2}\right)^2 + \left(x_{1}-2\right)^2.
\end{equation}
Functions $f_1$ and  $f_2$ have global minimizers at $x^*=(1,1)$ and $\hat{x}=(2,4)$, respectively. Note that $f_1(x^*)=f_2(\hat{x})=0$ and $f_1(\hat{x})=f_2(x^*)=1$. Figure~\ref{fig:image}(a) shows a representation of the image set of $F(x)$ {\it around} the Pareto front, obtained by discretizing the square $[-5,5]\times[-5,5]$ by a fine grid and plotting all the image points. We run the algorithm 1000 times using starting points from a uniform random distribution belonging to $(-5,5)\times(-5,5)$. In all instances, the Algorithm~\ref{sec:gradient} stopped at a point satisfying the convergence criterion. Figure~\ref{fig:image}(b) shows the image set of all final iterates. Thus, given a reasonable number of starting points, Algorithm~\ref{sec:gradient} was able to estimate the Pareto front of the considered Rosenbrock's problem. The value space generated by the Riemannian gradient method using others 200 random starting points with image belonging to the box $(0,4)\times(0,4)$ can be seen in Figure~\ref{fig:path}(a). A full point represents a final iterate whereas the beginning of a straight segment represents the corresponding starting point.

 \begin{figure}[h!]
\begin{minipage}[b]{0.50\linewidth}
\begin{figure}[H]
	\centering
		\includegraphics[scale=0.50]{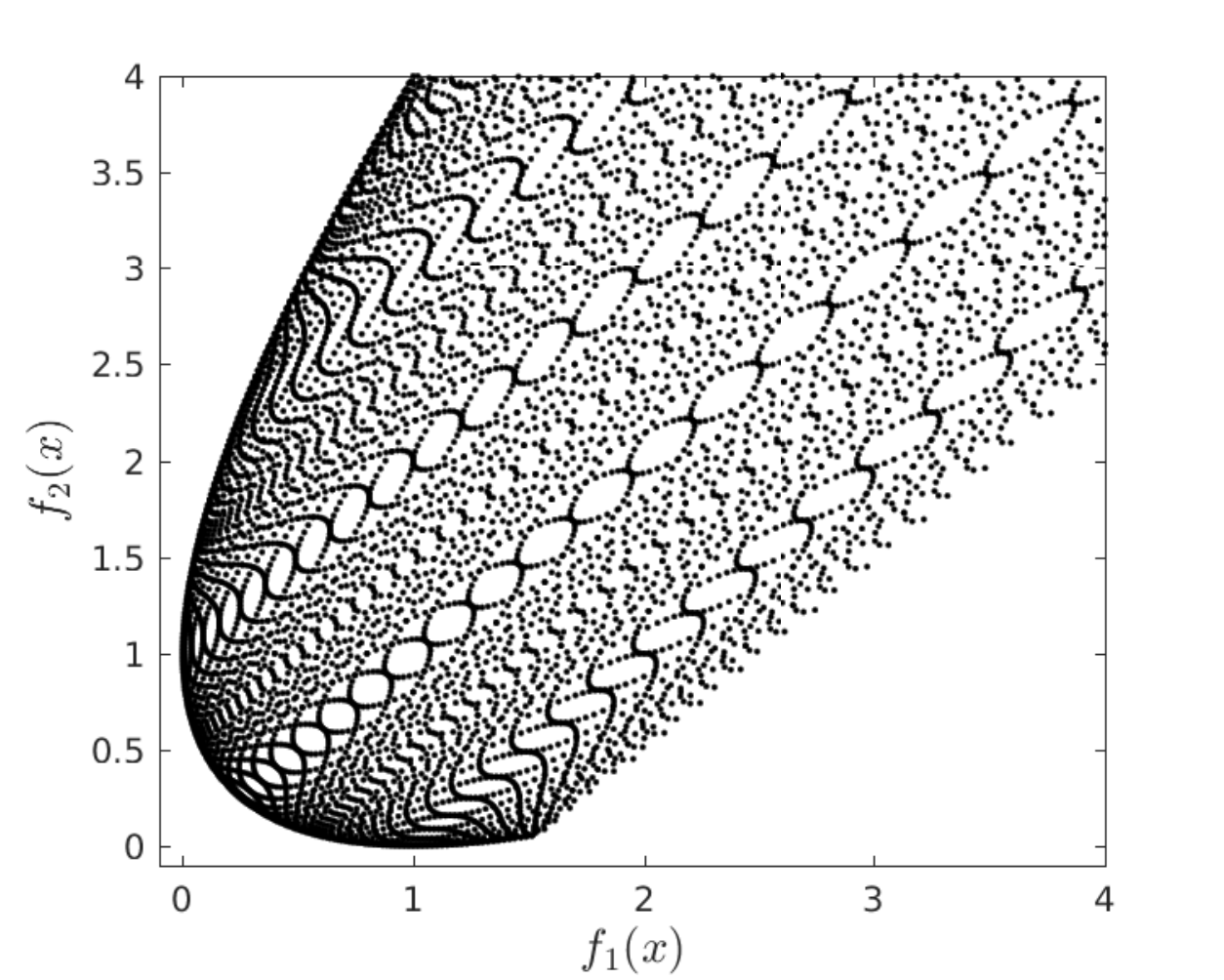}\\
	\footnotesize	(a)
\end{figure}

\end{minipage} \hfill
\begin{minipage}[b]{0.50\linewidth}

\begin{figure}[H]
	\centering
		\includegraphics[scale=0.50]{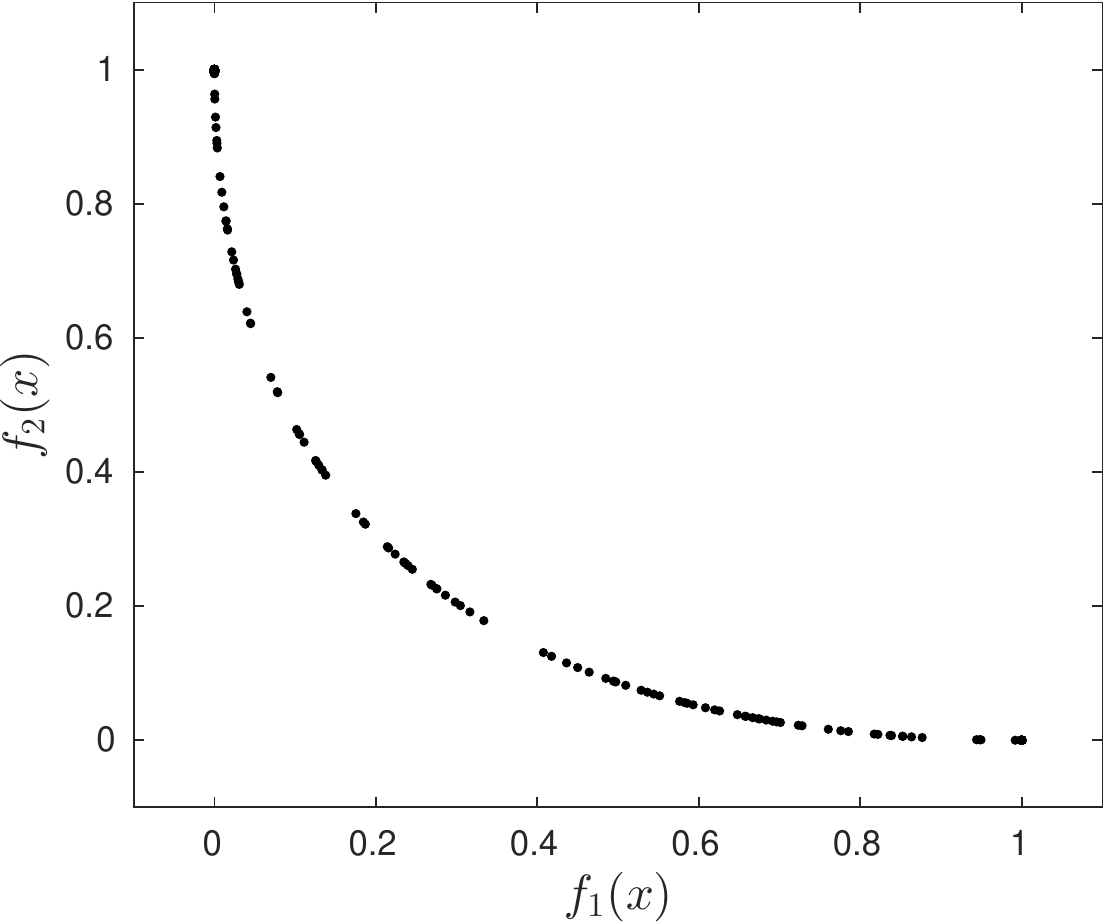}\\
	\footnotesize	(b)
\end{figure}
\end{minipage}\hfill
\caption{(a): Image set of the Rosenbrock's problem {\it around} its Pareto front; (b) value space of Rosenbrock's problem for 1000 random starting points belonging to $(-5,5)\times(-5,5)$.}
\label{fig:image}
\end{figure}

For comparative purposes, we implemented and tested the Euclidean gradient method for minimizing \eqref{eq:ros1}--\eqref{eq:ros2}. In summary, the Euclidean method corresponds to Algorithm~\ref{sec:gradient} with the usual inner product and the exponential map given by $\mbox{exp}_{x}(v)=x+v$. We point out that an equivalent Armijo-type line search employed in the Riemannian case was coded in the Euclidean algorithm. We also run the Euclidean algorithm using the same 1000 starting points belonging to $(-5,5)\times(-5,5)$ considered for the Riemannian algorithm. For each method, Table~\ref{tab:comparative} reports  the percentages of runs that has reached a critical point ($\%$) and, for the successful runs, the median of number of iterations (it), the median of functions evaluations (evalf), and the median of gradient evaluations (evalg). Thus, the reported data in Table~\ref{tab:comparative} represents a typical run of the Riemannian and the Euclidean algorithms. It is worth noting that we considered each evaluation of a coordinate function (resp. gradient) in the calculation of evalf (resp. evalg). Note that the number of steepest descent direction calculations is equal to the number of iterations.

 \begin{table}[h!] 
  \footnotesize
  \centering
\begin{tabular}{|c|c|c|c|c|} \hhline{~*4{-}|}  
\multicolumn{1}{c|}{} & \multicolumn{1}{c|}{\cellcolor[gray]{0.9} $\%$ } & \multicolumn{1}{c|}{\cellcolor[gray]{0.9} it }& \multicolumn{1}{c|}{\cellcolor[gray]{0.9} evalf }& \multicolumn{1}{c|}{\cellcolor[gray]{0.9} evalg }\\ \hline
\cellcolor[gray]{0.9} Riemannian method   & 100.0  & 5.0    & 49.0    & 12.0 \\ \hline
\cellcolor[gray]{0.9} Euclidean method    & 95.1   & 1629.0  & 5721.0  & 3260.0  \\ \hline
\end{tabular}
\caption{Performance of the Riemannian and Euclidean gradient methods in the Rosenbrock's problem.}
\label{tab:comparative}
\end{table}

As can be seen in Table~\ref{tab:comparative}, in the considered Rosenbrock's problem, the Riemannian algorithm is much superior to the Euclidean one. The introduction of a suitable metric that makes $F$ convex with componentwise Lipschitz continuous Jacobian enabled a huge reduction in computational cost to solve the problem. Figure~\ref{fig:path}(b) shows a typical behavior of the methods on the Rosenbrock's problem \eqref{eq:ros1}--\eqref{eq:ros2}. For each method, we plotted the image set of the generated sequence for the particular case where the starting point is $(0.5,0.2)$. The convergence criterion was satisfied with 25 and 1585 iterations for the Riemannian and Euclidean gradient methods, respectively. Due to the small steps sizes performed by the Euclidean method (typically of the order of $10^{-3}$), the corresponding path illustrated in the Figure~\ref{fig:path}(b) appears to be a continuous segment. In its turn, the Riemannian method quickly approaches the Pareto front.

 \begin{figure}[h!]
\begin{minipage}[b]{0.50\linewidth}
\begin{figure}[H]
	\centering
		\includegraphics[scale=0.50]{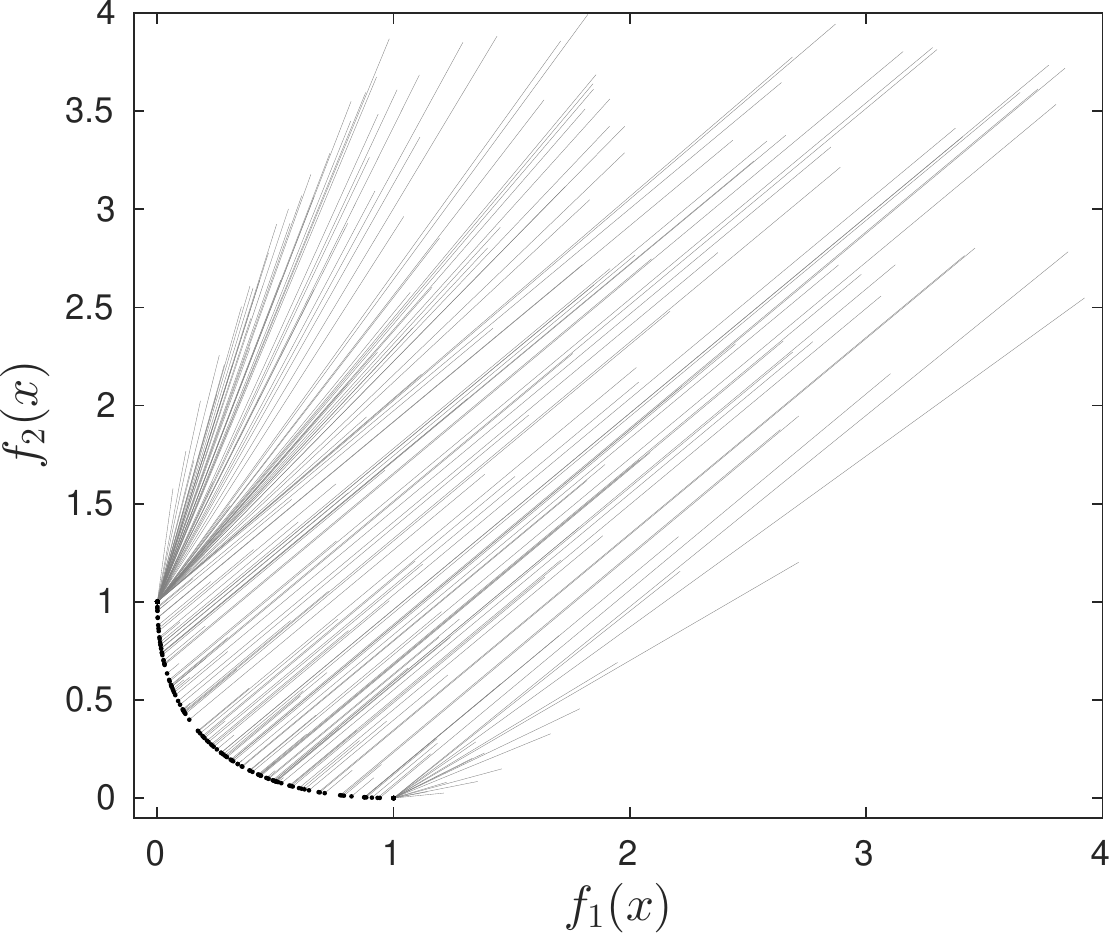}\\
	\footnotesize	(a)
\end{figure}

\end{minipage} \hfill
\begin{minipage}[b]{0.50\linewidth}

\begin{figure}[H]
	\centering
		\includegraphics[scale=0.50]{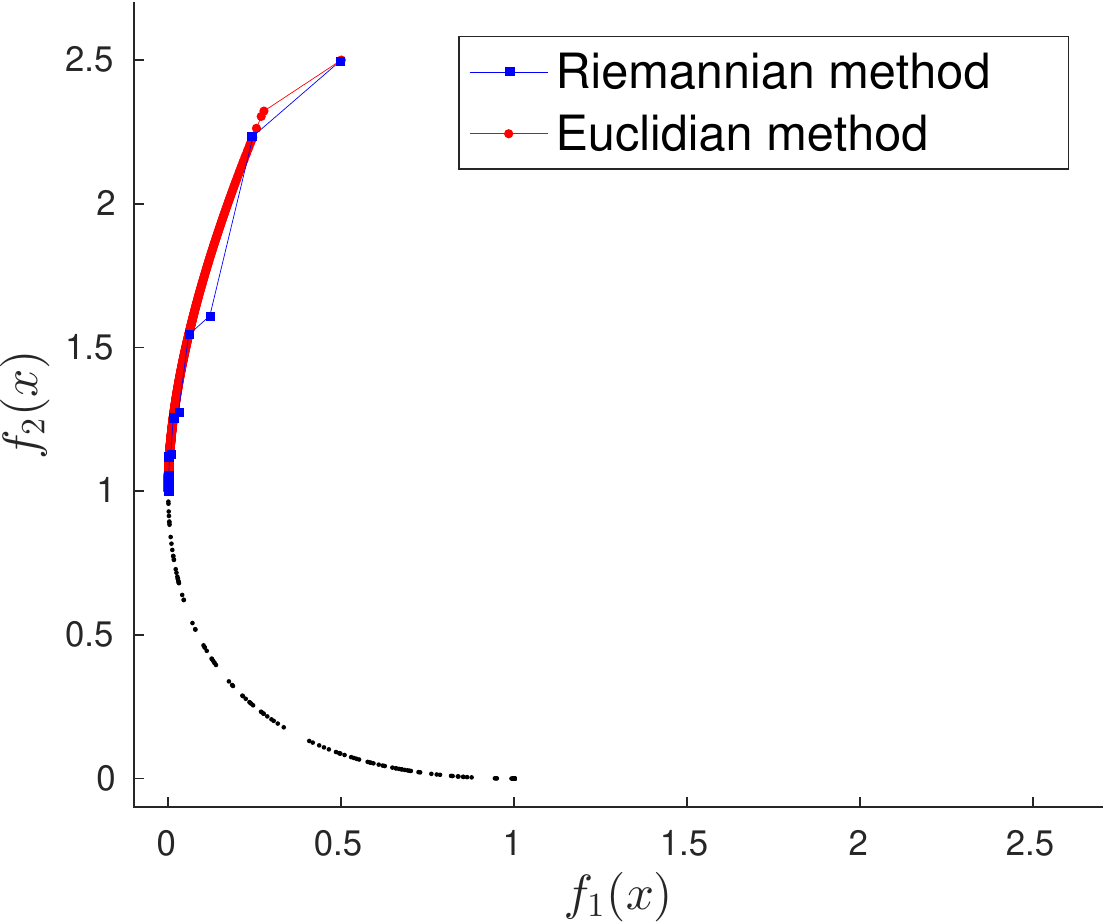}\\
	\footnotesize	(b)
\end{figure}
\end{minipage}\hfill
\caption{(a) Value space of Rosenbrock's problem for 200 starting points with image belonging to the box $(0,4)\times(0,4)$; (b) a typical behavior of the Riemannian and the Euclidean gradient methods on the Rosenbrock's problem.}
\label{fig:path}
\end{figure}

%
%
%

\subsection{Example in the Positive Orthant}

Now we consider the application of Algorithm~\ref{sec:gradient} for minimizing the vector function $F(x)=(f_1(x),\ldots,f_m(x))$ where $f_j(x)$ is given by \eqref{eq:psy}. Note that for the Riemannian manifold  ${\mathcal M}=(\mathbb{R}_{++}^n,G)$ and $x\in \mathbb{R}_{++}^n$, the tangent space $T_x{\mathcal M}$ corresponds to $\R^n$. Thus, problem~\eqref{vproblem} to calculate $v_x$ is directly posed as a quadratic programming problem.

Since in the previous section we solved only a small Rosenbrock's problem, we now consider larger instances of the problem related to Example~\ref{ex:fun.product}. First, we kept the number of objectives equal to two and varied the dimension of the space assigning the following values: $n = 10$, $100$, $400$, and $1000$. In the second set of tests, we set $n=100$ and varied the number of objectives taking $m = 10$, $20$, $100$, and $200$. All the parameters of each function $f_j$ in \eqref{eq:psy} were random generated belonging to $(0, 1)$. Each problem instance was solved 20 times using starting points from a uniform random distribution inside the box $(0, 10)^n$. The results in Table~\ref{tab:varying} are given in the same form as Table~\ref{tab:comparative}.

\begin{table}[h]{ \footnotesize \centering
\begin{minipage}[b]{0.50\linewidth}
\begin{tabular}{|c|c|c|c|c|c|}  \hline
\cellcolor[gray]{.9} $n$ &  \cellcolor[gray]{.9} $m$ & \cellcolor[gray]{.9} $\%$ & \cellcolor[gray]{.9} it & \cellcolor[gray]{.9} evalf & \cellcolor[gray]{.9} evalg \\ \hline
10   & 2 & 100.0 & 26.5  & 117.5  & 55.0  \\
100  & 2 & 100.0 & 71.5  & 220.0  & 145.0 \\
400  & 2 & 100.0 & 273.0 & 622.0 & 548.0 \\
1000 & 2 & 100.0 & 17.0  & 104.0  & 36.0  \\  \hline
\end{tabular}\\
\centering (a)
\end{minipage}\hfill
\begin{minipage}[b]{0.50\linewidth}
\begin{tabular}{|c|c|c|c|c|c|}  \hline
\cellcolor[gray]{.9} $n$ & \cellcolor[gray]{.9} $m$ & \cellcolor[gray]{.9} $\%$ & \cellcolor[gray]{.9} it & \cellcolor[gray]{.9} evalf & \cellcolor[gray]{.9} evalg \\ \hline
100 & 10  & 100.0 & 33.5 & 592.5   & 345.0  \\
100 & 20  & 100.0 & 34.0 & 1039.5  & 700.0  \\
100 & 100 & 100.0 & 24.5 & 3251.5  & 2550.0 \\
100 & 200 & 100.0 & 36.5 & 8864.0 & 7500.0 \\ \hline
\end{tabular}\\
\centering (b)
\end{minipage}\hfill
\caption{Performance of the Riemannian gradient method related to Example~\ref{ex:fun.product} varying: (a) the dimension of the space; (b) the number of objectives.}
\label{tab:varying}
}\end{table}


The highlight of Table~\ref{tab:varying} is that Algorithm~\ref{sec:gradient} was robust with respect to the dimension and to the number of objectives, which is consistent with the theoretical results. The results of the present section suggest that Algorithm~\ref{sec:gradient} is potentially able to solve large problems. Surprisingly, for the first set of problems, a fewer number of function/gradient evaluations were required for the case where $n = 1000$ compared to smaller instances of the problem. 

\subsection{Example in the Cone of Symmetric Positive Definite Matrices}
Let $\mathcal{M}$ be the Riemannian manifold $({\mathbb P}^n_{++}, \langle \cdot , \cdot \rangle)$, where the inner product is defined as in Example~\ref{ex:matrix1}. For $X \in \mathbb{P}_{++}^n$, the tangent space  $T_X{\mathcal M}$ corresponds to the set of the symmetric matrices ${\mathbb P}^n$. In our implementation, in order to compute the steepest descent direction, in addition to $\lambda$, the unknowns of problem~\eqref{vproblem} are the $(n^2+n)/2$ entries of the lower triangular part of the symmetric matrix $u$.

Given $X\in {\mathbb P}^n_{++}$ and $V \in {\mathbb P}^n$, direct calculations shows that the exponential map in \eqref{eq:expM} can be rewritten as $\exp_X(V)=X e^{X^{-1}V}$. For computing the inverse of matrix $X$, we used the LAPACK routine {\texttt dpotri} which uses the Cholesky factorization of $X$. For computing matrix exponentials, we used {\texttt dgpadm} routine of EXPOKIT package~\cite{sidje1998expokit}. It should be noted that {\texttt dpotri} and {\texttt dgpadm} are dense routines. 

We considered bicriteria and three-criteria problem instances related to Example~\ref{ex:matrix1}. The parameters of function \eqref{eq:exmatrix} were randomly generated belonging to $(0,1)$. For each instance, we run the Riemannian gradient method 20 times using random starting points with eigenvalues belonging to the interval $(0, 100)$. The results in Table~\ref{tab:matrix} show that Algorithm~\ref{sec:gradient} solved all the instances with a moderate computational effort. It is worth mentioning that in a typical iteration, the first trial step size of Strategy~\ref{armijo.step} defined by
\begin{equation}\label{eq:firstrial}
{\hat t}_{k_0} = \max \{ t_{\min}, \min \{ {\bar t}_{k_0},t_{\max} \} \}, \qquad
{\bar t}_{k_0} = \left\{\begin{array}{ll}
                1/\left\|v_0\right\|, & \mbox{ if } k=0, \\
               t_{k-1}  \left\|v_{k-1}\right\|^2/\left\|v_k\right\|^2, & \mbox{ if } k\geq 1,\\     
               \end{array}\right.
\end{equation}
satisfies the sufficient descent condition \eqref{eq:sl}. Indeed, as it can be seen Table~\ref{tab:matrix}, the values reported in evalf columns are slightly greater than the corresponding number of iterations times the number of objectives $m$. We observe that the choice \eqref{eq:firstrial} corresponds to the safeguarded Shanno and Phua~\cite{shanno} recommendation and was first proposed in the multiobjective optimization setting in \cite{PerezPrudente2018}.

\begin{table}[h]{ \footnotesize \centering
\begin{minipage}[b]{0.50\linewidth}
\begin{tabular}{|c|c|c|c|c|c|}  \hline
\cellcolor[gray]{.9} $n$ & \cellcolor[gray]{.9} $m$ & \cellcolor[gray]{.9} $\%$ & \cellcolor[gray]{.9} it & \cellcolor[gray]{.9} evalf & \cellcolor[gray]{.9} evalg \\ \hline
5     & 2 & 100.0 &   8.0 &  26.5 &  18.0 \\
10   & 2 & 100.0 &   13.0 &  38.5 & 28.0  \\
20   & 2 &100.0 &  18.0 & 49.0 & 38.0  \\
50   & 2 & 100.0 &  27.0 & 64.0 & 56.0 \\ \hline
\end{tabular}\\
\centering (a)
\end{minipage}\hfill
\begin{minipage}[b]{0.50\linewidth}
\begin{tabular}{|c|c|c|c|c|c|}  \hline
\cellcolor[gray]{.9} $n$ & \cellcolor[gray]{.9} $m$  & \cellcolor[gray]{.9} $\%$ & \cellcolor[gray]{.9} it & \cellcolor[gray]{.9} evalf & \cellcolor[gray]{.9} evalg \\ \hline
5     & 3 &100.0 & 7.0  & 28.5  & 24.0  \\
10   & 3 &100.0 & 12.0  & 45.5   & 39.0  \\
20   & 3 & 100.0 &   18.0 & 68.0 & 57.0  \\
50   & 3 &100.0 &   28.0 & 91.0 & 87.0 \\ \hline
\end{tabular}\\
\centering (b)
\end{minipage}\hfill
\caption{Performance of the Riemannian gradient method related to Example~\ref{ex:matrix1} for: (a) bicriteria problems; (b) three-criteria problems.}
\label{tab:matrix}
}\end{table}


Finally, we report that Algorithm~\ref{sec:gradient} converges with a single iteration when applied to instances of Example~\ref{ex:matrix2}. The considered metric makes it possible to explore the structure of the problem turning it into a trivial problem from the Riemannian perspective.
 
\section{Conclusions}\label{conclusion}
In this paper,  the behavior of  the  steepest descent method for multiobjective optimization on Riemannian manifolds  with lower bounded sectional curvature is analyzed.  It  would be interesting to study stochastic versions of this  method.  An  interesting question to be also  investigated is the extension  and analysis of subgradient method in  this new setting.


%
%




\end{document}